\documentclass[review]{elsarticle}
\usepackage{amsfonts}
\usepackage{amsmath}
\usepackage{amsthm}
\usepackage{amssymb}
\usepackage{mathtools}
\usepackage{hyperref}
\usepackage{lineno}
\usepackage{graphicx}
\usepackage{comment}

\newtheorem{theorem}{Theorem}
\newtheorem{lemma}{Lemma}
\newtheorem{assumption}{Assumption}
\newtheorem{corollary}{Corollary}
\newdefinition{rmk}{Remark}[section]

\let\det\relax
\DeclarePairedDelimiter{\Fnorm}{\lVert}{\rVert_F}
\DeclarePairedDelimiter{\FnormSq}{\lVert}{\rVert_F^2}
\DeclarePairedDelimiter{\norm}{\lVert}{\rVert}
\DeclarePairedDelimiter{\Twonorm}{\lVert}{\rVert_2}
\DeclarePairedDelimiter{\TwonormSq}{\lVert}{\rVert_2^2}
\DeclarePairedDelimiter{\TwoInfnorm}{\lVert}{\rVert_{2,\infty}}
\DeclarePairedDelimiter{\TwoOnenorm}{\lVert}{\rVert_{2,1}}
\DeclarePairedDelimiter\det{|}{|}

\newcommand{\bx}{\boldsymbol{x}}
\newcommand{\bz}{\boldsymbol{z}}
\newcommand{\bepsilon}{\boldsymbol{\varepsilon}}
\newcommand{\bX}{\boldsymbol{X}}
\newcommand{\bY}{\boldsymbol{Y}}
\newcommand{\bXn}{\boldsymbol{X_n}}
\newcommand{\bYn}{\boldsymbol{Y_n}}
\newcommand{\bepsilonn}{\boldsymbol{\varepsilon_n}}
\newcommand{\bBzero}{\boldsymbol{B_0}}
\newcommand{\bB}{\boldsymbol{B}}
\newcommand{\bBn}{\boldsymbol{B_n}}
\newcommand{\bSigma}{\boldsymbol{\Sigma}}
\newcommand{\bzero}{\boldsymbol{0}}

\newcommand{\Nor}{\mathcal{N}}

\newcommand{\InvWis}{\mathcal{IW}}
\newcommand{\Ball}{\mathcal{B}}

\newcommand{\InvbSigma}{\bSigma^{-1}}
\newcommand{\inddist}{\overset{ind}{\sim}}
\newcommand{\bigO}{\mathcal{O}}
\newcommand{\sqrtbSigma}{\bSigma^{1/2}}
\newcommand{\sqrtbSigmaO}{\bSigma_0^{1/2}}
\newcommand{\sqrtbSigmaInv}{\bSigma^{-1/2}}
\newcommand{\sqrtbSigmaOInv}{\bSigma_0^{-1/2}}

\newcommand{\Data}{\mathcal{D}_n}
\newcommand{\E}{\mathbb{E}}

\newcommand{\RomanNumeralCaps}[1]{\MakeUppercase{\romannumeral #1}}

\begin{document}

\begin{frontmatter}
	
	\title{Ultra High-dimensional Multivariate Posterior Contraction Rate under Shrinkage Priors}
	
%

\author[1]{Ruoyang Zhang}
\ead{njiandan@ufl.edu}
\author[1]{Malay Ghosh}
\ead{ghoshm@ufl.edu}

\address[1]{Department of Statistics, University of Florida, Gainesville, FL 32611, United States}

\begin{abstract}
In recent years, shrinkage priors have received much attention in high-dimensional data analysis from a Bayesian perspective. Compared with widely used spike-and-slab priors, shrinkage priors have better computational efficiency. But the theoretical properties, especially posterior contraction rate, which is important in uncertainty quantification, are not established in many cases. In this paper, we apply global-local shrinkage priors to high-dimensional multivariate linear regression with unknown covariance matrix. We show that when the prior is highly concentrated near zero and has heavy tail, the posterior contraction rates for both coefficients matrix and covariance matrix are nearly optimal. Our results hold when number of features p grows much faster than the sample size n, which is of great interest in modern data analysis. We show that a class of readily implementable scale mixture of normal priors satisfies the conditions of the main theorem.
\end{abstract}


\end{frontmatter}

\section{Introduction}
Parameter estimation, variable selection and prediction in high dimensional regression models have received significant attention in these days, particularly when the number of regressors $p$ is much larger than the number of observations $n$. Examples abound - brain imaging, microarray experiments, satellite data analysis, just to name a few. In many of these examples, one key issue is to address sparsity of effective regression parameters in the midst of a multitude of inactive ones. For example, there are only a few significant genes associated with Type I diabetes along with million others of no direct impact for such a disease.

In a frequentist framework, the most commonly used approach for inducing sparsity is by imposing regularization penalty on the parameters of interest. The most popular ones are $\ell_1$ (lasso) and $\ell_2$ (ridge) penalties or a combination of these (elastic net). The $\ell_1$ and $\ell_2$ regularization can naturally be extended to multivariate case where sparsity in the coefficient matrix is desired. 
\citet{rothman2010sparse} used $\ell_1$ penalties on each entry of the coefficient matrix as well as on each off-diagonal element of the covariance matrix. 
\citet{wilms2018algorithm} considered a model which put an $\ell_2$ penalty on the rows of coefficient matrix to shrink the entire row to zero, and an $\ell_1$ penalty on the off-diagonal elements of the inverse error covariance matrix. 
\citet{li2015multivariate} proposed a multivariate sparse group lasso imposing $\ell_2$ penalty on the rows of the regression matrix and in addition an $\ell_1$ penalty on individual coefficient of the regression matrix to perform sparse estimation and variable selection both at the between and within group levels. 

In a Bayesian setting, spike-and-slab priors, originally introduced by \citet{mitchell1988bayesian} have become very popular for handling sparsity. Spike-and-slab priors are mixture densities with positive mass at zero to force some parameters to be zero, and a continuous density to model the nonzero coefficients. These priors have been used in a variety of contexts. 
For example, for Bayesian Group Lasso, \citet{xu2015bayesian} used these priors for both variable selection and estimation. This work was extended by \citet{liquet2017bayesian} to the multivariate case. 
More recently, \citet{rovckova2018spike} introduced spike-and-slab lasso for variable selection and estimation. \citet{deshpande2019simultaneous} extended it to multivariate case by putting spike-and-slab prior on each entry of the coefficient matrix as well as on each off-diagonal element of the precision matrix.

Spike-and-slab priors face severe computational challenges, when $p$, the number of regressors, is very large. This is due to the fact that one needs to search over $2^p$ possible models. \citet{bai2018high} provided an alternative to these priors by introducing global-local shrinkage priors. These priors approximate the spike-and-slab priors well and are usually much easier to implement because they are continuous. Like spike-and-slab priors, global-local shrinkage priors also put significant probability around zero, but retain heavy enough tails so that the true signals are very unlikely to be missed. 

\citet{bai2018high} considered the case when the number of regressors can grow at a sub-exponential rate when compared to the sample size. They established posterior consistency of their prior and showed that the insignificant regression coefficients converge to zero at an exponential rate. \citet{song2017nearly} provided some general posterior contraction rates in the context of variable selection and estimation in univariate regression models with unknown variance. 

Our paper is a follow-up of the works by \citet{bai2018high} and \citet{song2017nearly}. In particular, unlike the former, we do not need to assume a known covariance matrix in the original regression model to establish exponential convergence rate of tail probabilities. We propose a set of general conditions on continuous prior for achieving nearly-optimal posterior contraction rate for both coefficient matrix and covariance matrix. This extends the work of \citet{song2017nearly} to the multivariate case. Also,  we have demonstrated that these regulatory conditions are satisfied by a general class of global-local shrinkage priors. Our technical results borrowed tools developed by \citet{song2017nearly}, but handling multivariate data presented some new challenges in proving the results.

\citet{ning2018bayesian} also addressed the issue of variable selection with unknown covariance matrix and established posterior consistency result similar to ours. But their results are based on spike-and-slab priors instead of global-local shrinkage priors and utilized different techniques from ours.

This paper is organized as follows. In Section 2, we establish general conditions on priors for achieving nearly-optimal posterior contraction rate for both coefficient matrix and covariance matrix. In Section 3, a class of global-local shrinkage prior that satisfies these general conditions is proposed. In Section 4, finite sample performance of the proposed model is evaluated through numerical experiments. Some final remarks are made in Section 5. Most of the technical theorems and lemmas are relegated to the Appendix.

\section{Posterior Contraction Rate} 
\subsection{Problem Setting}
We consider the following multivariate linear regression model
\begin{equation}
\label{regmodel}
Y_i = X_i\bBn + \varepsilon_i\sqrtbSigma_n, \quad i=1,\cdots,n
\end{equation}
where $Y_i$ is a $1\times q_n$ response vector, and the correlation of responses is assumed to be captured by the $q_n \times q_n$ covariance matrix  $\bSigma_n$. $\bBn$ is a $p_n\times q_n$ coefficient matrix, $X_i$ is a $1\times p_n$ regressor vector, $\varepsilon_i$ is a $1\times q_n$ noise vector. Throughout this paper, $\varepsilon_i$'s are assumed to have i.i.d multivariate normal $\Nor(0,I_{q_n})$ distribution, $i=1,\cdots,n$. Subscripts $n$ denotes that the quantity can vary with $n$. In matrix form, Model \eqref{regmodel} can be written as
\begin{equation}
\label{regmodel_mat}
\bYn = \bXn\bBn + \bepsilonn\sqrtbSigma_n
\end{equation}
where $\bYn = (Y_1^T,\cdots,Y_n^T)^T$, $\bXn = (X_1^T,\cdots,X_n^T)^T$ and 
$\bepsilonn = (\varepsilon_1^T,\cdots,\varepsilon_n^T)^T$.
\\
Throughout the paper,  for notational simplicity, subscript $n$ for $\bY_n$, $\bX_n$ and $\bB_n$ will be dropped when there is no ambiguity.

For estimation of $\bB$ and $\bSigma$, we consider the following Bayesian multivariate linear regression model. This model puts independent prior on each row vector of $\bB$ conditioning on $\bSigma$ and an Inverse-Wishart prior for $\bSigma$. General conditions for $\pi(\bB|\bSigma)$ for establishing a satisfying posterior contraction rate of  $\bB$ and $\bSigma$ is given in Theorem \eqref{thm:posterior contraction}.
\begin{align}\label{general model}
\begin{split}
Y_i | X_i, \bB, \bSigma & \inddist \Nor_{q_n}(X_i\bB, \bSigma)  \quad i=1,\cdots,n \\
\bB_j | \bSigma &\inddist \pi(\bB_j|\bSigma) \quad j=1,\cdots,p_n\\
\bSigma & \sim \InvWis_{q_n}(\upsilon, \Phi)
\end{split}
\end{align}
where $\bB_j$ is the $jth$ row of $\bB$. $\InvWis_{q_n}(\upsilon, \Phi)$ means a $q_n$-dimensional Inverse-Wishart distribution with degree of freedom $\upsilon>q_n-1$ and a $q_n \times q_n$ positive definite scale matrix $\Phi$.

\subsection{Notations}
First, a few notations used throughout the paper are defined. We write $a\lor b$ for $\max(a,b)$, where $a$ and $b$ are real numbers. Letters $C, c, k$ with subscripts denote generic positive constants that do not depend on $n$. For two sequences of positive real numbers $a_n$ and $b_n$, $a_n \lesssim b_n$ is equivalent to $a_n =\bigO(b_n)$, i.e. there exists constant $C>0$ such that $a_n \leq C b_n$ for all large $n$. $a_n \prec b_n$ means $a_n =o(b_n)$, that is, $a_n/b_n \to 0$ as $n \to \infty$. $a_n \simeq b_n$ denotes that there exists constants $0<C_1\leq C_2$ such that $C_1 b_n \leq a_n \leq C_2 b_n$.

For a vector $\bx \in \mathbb{R}^p$, $\Twonorm{\bx}$ denotes the $\ell_2$ norm. For a $n\times m$ real matrix $A$ with entries $a_{ij}$, $\Fnorm{A} \coloneqq \sqrt{tr(A A^T)}$ denotes the Frobenius norm of $A$; $\TwoInfnorm{A} \coloneqq \max_{1\leq i \leq n} \big(\sum_{j=1}^{m} a_{ij}^2 \big)^{1/2}$ denotes $A$'s maximum row length; $\TwoOnenorm{A} \coloneqq \sum_{i=1}^{n}\big(\sum_{j=1}^{m} a_{ij}^2 \big)^{1/2}$ denotes the sum of row lengths. For a symmetric real matrix $A$, $\lambda_i(A)$ denotes the $ith$ smallest eigenvalue of $A$. $\norm{A} = \lambda_{\max}(A)$ denotes the spectral norm of $A$, which is also the maximum eigenvalue of $A$.

\subsection{ Conditions for Posterior Contraction Rate}
Suppose the data $\bY$ is generated by \eqref{regmodel} with the true regression parameter $\bBzero$ and the true dispersion matrix $\bSigma_0$. To achieve posterior contraction rate, we first state some assumptions for sparsity of $\bBzero$, the eigen-structure of design matrix $\bX$, and eigenvalues of $\bSigma_0$.

\begin{assumption} \label{assumption:beta}
	Sparsity of $\bBzero$:\\
	$A_1$: $s_0\log p_n \prec n$, where $s_0$ is the size of the true model, i.e., the number of nonzero rows in $\bBzero$.
\end{assumption}

\begin{assumption} \label{assumption:X}
	Eigen-structure of the design matrix $\bX$:\\
	$A_2(1)$: Entries $(\bX)_{ij}$ in design matrix are uniformly bounded. For simplicity, assume they are bounded by 1.\\
	$A_2(2)$: $p_n \to \infty$ as $n\to\infty$.\\
	$A_2(3)$: There exist some integer $\bar{p}$(depending on $n$ and $p_n$) and fixed constant $\lambda_0$ such that $\bar{p}\succ s_0$, and $\lambda_{\min}(X_S^T X_S)\geq n\lambda_0$ for any subset model $S$ with $\det{S}\leq\bar{p}$.
\end{assumption}

\begin{assumption} \label{assumption: bSigma}
	Dimension and eigenvalues of $\bSigma_0$:\\
	$A_3(1)$: $q_n\lesssim \log p_n$.\\
	$A_3(2)$: $q_n^2\log n\prec n$.\\
	$A_3(3)$: $b_1\leq \lambda_i(\bSigma_0) \leq b_2, \quad i=1,\cdots, q_n.$
\end{assumption}

\begin{rmk}
Assumption \eqref{assumption:X} and Assumption \eqref{assumption:beta} are the same as in \cite{song2017nearly}. Note that $A_2(2)$ does not restrict the rate of $p_n$ going to infinity. Along with $A_1$, $p_n$ can grow sub-exponentially fast with $n$ when $s_0$ is finite, e.g., $\log(p_n)\lesssim n^c$ for some $c\in(0,1)$, which is the ultrahigh dimensional setting in \cite{bai2018high}. 
\end{rmk}

\begin{rmk}
$A_3(2)$ and $A_3(3)$ are the same as in \cite{ning2018bayesian}. Different from many previous settings where the dimension of response $q_n$ is a fixed constant (\cite{bai2018high}, \cite{liquet2017bayesian}), here we allow $q_n$ to grow with $n$. However, the growth of $q_n$ is limited by constraints $A_3(1)$ and $A_3(2)$. When $q_n$ is a fixed constant, $A_3(1)$ and $A_3(2)$ are trivially satisfied.
\end{rmk}

\begin{theorem} \label{thm:posterior contraction}
	For the multivariate Bayesian model given in \eqref{general model},  suppose design matrix $\bX$ satisfies Assumption \eqref{assumption:X} and true parameter $(\bBzero,\bSigma_0)$ satisfies Assumptions \eqref{assumption:beta} and \eqref{assumption: bSigma}. Let the prior density of $\bB$ be:
	\[
	\pi(\bB|\bSigma)=\prod_{i=1}^{p_n} \{\det{\bSigma}^{-1/2}g_{\tau}(\bB_j\bSigma^{-1/2})\}
	\] 
	If $g_\tau(\bx)$ satisfies
	\begin{align}
	&\int_{\Twonorm{\bx} \geq a_n} g_\tau(\bx)d\bx \leq p_n^{-(1+u)} \mbox{ for some } u>0, \label{cond:highlynear0}\\
	&\log\big(\inf_{\Twonorm\bx \leq M_0} g_\tau(\bx)\big) \gtrsim -\log p_n \label{cond:fattail}
	\end{align}
	where $M_0=\gamma\TwoInfnorm{\bBzero\sqrtbSigmaOInv}, \; \gamma>1$, $a_n \simeq \epsilon_n/p_n$.\\
	
	Then the following posterior contraction result holds
	\[
	\Pi_n( \Fnorm{(\bB-\bBzero)\bSigma_0^{-1/2}} \geq M \epsilon_n | \bY) \to 0 \mbox{ in } P_{(\bBzero, \bSigma_0)}\mbox{-probability}
	\]
	\[
	\Pi_n( \norm{\bSigma-\bSigma_0} \geq M\norm{\bSigma_0}\epsilon_n | \bY) \to 0 \mbox{ in } P_{(\bBzero, \bSigma_0)}\mbox{-probability}
	\]
	where $\epsilon_n = \sqrt{s_0\log p_n/n} \lor \sqrt{q_n^2\log n/n} \lor \sqrt{q_ns_0\log n/n}$ and $M$ is a sufficiently large constant.
\end{theorem}

\begin{rmk}
Conditions \eqref{cond:highlynear0}) and \eqref{cond:fattail} for $g_\tau(\cdot)$ have intuitive interpretation. \eqref{cond:highlynear0}  means that the prior has to be highly concentrated around a small neighborhood of $\mathbf{0}$, which corresponds to the sparsity structure of the model. Taking $\TwoInfnorm{\bBzero\sqrtbSigmaOInv}$ as the strength of true signal, \eqref{cond:fattail} means that the prior needs to put enough mass around the true signal, which is often referred as heavy-tail condition in \cite{polson2010shrink, armagan2013posterior, armagan2013generalized}. 
\end{rmk}

\begin{rmk}
When $q_n$ is a fixed constant, the contraction rate $\epsilon_n$ becomes $\sqrt{s_0\log p_n/n}$, which is the same as the univariate optimal posterior contraction rates for regression coefficient with respect to $\ell_1$ and $\ell_2$ norm in \cite{castillo2015bayesian, rovckova2018bayesian}, where spike-and-slab priors are used.  In addition, this rate is also comparable to the minimax rate $\sqrt{s_0\log (p_n/s_0)/n}$ of lasso and Dantzing selector for $\ell_2$ loss in $\ell_0$ ball \cite{raskutti2011minimax, ye2010rate}. Two additional terms $\sqrt{q_n^2\log n/n}$ and $\sqrt{q_ns_0\log n/n}$ that may slower the convergence can be viewed as a compensation of allowing  $q_n\to\infty$. 
\end{rmk}

\begin{rmk}
 By the fact that $\Twonorm{\bSigma_0} \leq \Fnorm{\bSigma_0}$ and $\Fnorm{\bSigma-\bSigma_0} \leq \sqrt{q_n}\Twonorm{\bSigma-\bSigma_0}$, we get $\Pi_n( \Fnorm{\bSigma-\bSigma_0} \geq M\Fnorm{\bSigma_0}\sqrt{q_n}\epsilon_n | Y) \to 0 \mbox{ in } P_{(\bBzero, \bSigma_0)}\mbox{-probability}$. Further, if $q_n$ is a constant, the posterior contraction rate of $\bSigma$ under Frobenius norm is also $\sqrt{s_0\log p_n/n}$.
\end{rmk}

The complete proof of Theorem \eqref{thm:posterior contraction} is provided in Appendix. Here we briefly summarize the ideas and key steps. We applied the tools developed in \cite{song2017nearly}. To extend univariate contraction results to multivariate case, spectral norm is used for measuring matrix distance. With its relation to Frobenius norm, we are able to make straightforward interpretations. 

For showing the posterior contraction results, auxiliary sets $A_n$, $B_n$ and $C_n$ are constructed as follow.  
\begin{equation*}
\begin{aligned}
A_n = &\{\mbox{at least } \tilde{p} \mbox{ entries } \norm{\bB_j\bSigma^{-1/2}} \mbox{ is larger than } a_n\} \\
\cup  &\{\norm{\bSigma-\bSigma_0} \geq M(\norm\bSigma \lor \norm{\bSigma_0})\epsilon_n \} \\
\cup  &\{\Fnorm{(\bB-\bBzero)\bSigma_0^{-1/2}} \geq M \epsilon_n \}
\end{aligned}
\end{equation*}

Define $B_n = \{\mbox{at least } \tilde{p} \mbox{ entries }  \Twonorm{\bB_j\bSigma^{-1/2}} \mbox{ is larger than } a_n\}$,  
and $C_n=A_n \setminus B_n$. Let $\theta = (\bB,\bSigma)$ and $\theta_0 = (\bBzero,\bSigma_0)$. It suffices to show $P_{\theta_0}(\Pi_n(A_n|\bY)\geq  e^{-\tilde{c_1} n\epsilon_n^2} \big) \leq e^{-\tilde{c_2} n\epsilon_n^2}$. By Lemma A.4 in \cite{song2017nearly}, the proof is composed of three parts: \\
(1) Construction of test $\phi_n$ satisfies $\E_{\theta_0} \phi_n \leq e^{-k_2 n\epsilon_n^2}$ and $\sup_{\theta \in C_n} \E_{\theta} (1-\phi_n) \leq e^{-k_3 n\epsilon_n^2}$. \\
(2) Showing event $B_n$ has very small probability under the specified prior. \\
(3) Demonstrating the marginal probability of data is highly likely to be bounded away from 0 if data is generated with true parameters. Probability bounds of Inverse Wishart distribution\cite{ning2018bayesian} are applied in this part.

\subsection{Variable Selection Consistency}
Different from spike-and-slab priors, continuous global-local shrinkage priors put zero probability at the point $\mathbf{0}$, so the solution is not sparse naturally. In this subsection, variable selection criteria and corresponding selection consistency property are discussed. But we want to point out that variable selection is not always required. Sometimes in practice, lacking exact zeros is deemed to be more realistic and preferred\citep{stephens2009bayesian}.

By Condition \eqref{cond:highlynear0} in Theorem \eqref{thm:posterior contraction}, where $a_n$ acts like a partition for ``spike" and ``slab" parts, the posterior model selection rule is set to be $S_n = \{j:\norm{\bB_j\bSigma^{-1/2}}>a_n\}$. Consistency of the selection rule is established in the following theorem. 

\begin{theorem} \label{thm:selection_consistency}
	Suppose assumptions and conditions for Theorem \eqref{thm:posterior contraction} hold with $a_n \prec \sqrt{\log p_n / n}/p_n$ and $u>1$ in \eqref{cond:highlynear0}. Let $\mathcal{B}_{j,\epsilon_n}:=Ball(\bB_j\sqrtbSigmaInv, c_0\epsilon_n)$, i.e., a $q_n$-dimensional ball centering at $\bB_j\sqrtbSigmaInv$ with radius $c_0\epsilon_n$, where $c_0>0$ is a constant. Let $S_n$ be the posterior subset model: $S_n=\{j:\norm{\bB_j\bSigma^{-1/2}}>a_n\}$.
	Suppose $\bBzero$ and $g_\tau(\bx)$ satisfies
	\begin{align}
	& \min_{j\in S_0} \norm{\bBzero_j} \geq M_1 \epsilon_n \text{ for some large constant } M_1, 
	\label{cond:min_signal} \\
	& s_0\log{l_n}\prec \log p_n \label{cond:flatness_g}
	\text{ where } l_n = \max_{j\in{S_0}} \sup_{\bx_1, \bx_2 \in \mathcal{B}_{j,\epsilon_n}}\dfrac{g_\tau(\bx_1)}{g_\tau(\bx_2)},
	\end{align}
	
	then $\Pi_n(S_n = S_0|\bY) \to 1$ in $P_{(\bBzero, \bSigma_0)}$.
\end{theorem}

\begin{rmk}
Compared to Theorem \eqref{thm:posterior contraction}, this theorem holds with a more concentrated prior peak implied by smaller $a_n$ and larger $u$. In addition, condition \eqref{cond:min_signal} requires minimal strength of the true coefficients. Intuitively, if a true parameter is too small, it would be hard to distinguish it from zero. In condition \eqref{cond:flatness_g}, $l_n$ can be viewed as a measurement of ``flatness" of $g_\tau(\bx)$ around true coefficients. With enough mass around the truth and this flatness constraint, we can get sufficiently large prior density for points in a small neighborhood of the true parameters. With stronger conditions than Theorem \eqref{thm:posterior contraction}, this result also gives a stronger posterior contraction that the false coefficients are bounded by $a_n$. As $a_n \to 0$ when $n \to \infty$, the false coefficients will diminish to zero in the limit. 
\end{rmk}

\section{Extended MBSP Model with Unknown Covariance Matrix}\label{EMBSP}
In previous section, we establish general conditions on the priors to obtain good posterior contraction and variable selection. In this section, we will propose a class of global-local shrinkage prior that satisfies conditions \eqref{cond:highlynear0} and \eqref{cond:fattail}. 

This class of priors we propose is scale mixture of Gaussians, which is closely related to the Multivariate Bayesian model with Shrinkage Priors (MBSP) introduced by \cite{bai2018high}. In MBSP, $\bSigma$ is assumed to be fixed and known. Here we put an Inverse-Wishart prior for $\bSigma$, extending it to the unknown $\bSigma$ case and obtain the following Extended MBSP model:
\begin{align}\label{model: Extended MBSP}
\begin{split}
\bB_j | \xi_j, \bSigma &\inddist \Nor_{q_n}(0, \tau_n\xi_j\bSigma) \quad j=1,\cdots,p_n \\
\xi_j &\inddist \pi(\xi_j)\\
\bSigma & \sim \InvWis_{q_n}(\upsilon, \Phi)
\end{split}
\end{align}

In univariate case($q_n=1$), many priors can be expressed as scale mixtures of Gaussians\cite{Tang2018}. Table \eqref{table: mixingcomponent}  lists such priors and corresponding mixing density $\pi(\xi)$. 
\begin{table}[!htbp]
	\centering
	\caption{List of scale mixtures of Gaussian priors}
	\label{table: mixingcomponent}
	\begin{tabular}{ll}	
		\hline
		prior       & $\pi(\xi)$  \\ \hline 
		Student's t &  $\xi^{-a-1}\exp(-a/\xi)$  \\ 
		TPBN \cite{armagan2011generalized}       &  $\xi^{u-1}(1+\xi)^{-a-u}$  \\
		Horseshoe   \cite{carvalho2010horseshoe} &  $\xi^{-1/2}(1+\xi)^{-1}$     \\
		NEG   \cite{griffin2010inference}      &  $(1+\xi)^{-a-1}$        \\ 
		GDP	\cite{armagan2013generalized}		& 
		$\int_{0}^{\infty}\frac{\lambda^2}{2}\exp(-\frac{\lambda^2\xi}{2})\lambda^{2a-1}\exp(-\eta\lambda)  d\lambda $ \\ 
		HIB	\cite{polson2012half}	& $\xi^{u-1}(1+\xi)^{-a-u}\exp(-\frac{s}{1+\xi})(\phi^2+\frac{1-\phi^2}{1+\xi})^{-1} $ \\ 
		Horseshoe+ \cite{bhadra2017horseshoe+}  & $\xi^{-1/2}(\xi-1)^{-1}\log\xi$    \\ 
		\hline
	\end{tabular}
\end{table}

We now show that when the mixing component $\pi(\xi)$ follows certain polynomial-tailed distribution, posterior contraction is obtained with proper global shrinkage parameter $\tau_n$.

\begin{theorem}\label{thm: scalemixureprior}
Suppose $\bB$ follow the following prior:
\begin{align}
\begin{split}
\bB_j | \xi_j, \bSigma &\inddist \Nor_{q_n}(0, \tau_n\xi_j\bSigma) \quad j=1,\cdots,p_n \\
\xi_j &\inddist \pi(\xi_j)
\end{split}
\end{align}
where $\pi(\xi_j)$ is a polynomial-tailed distribution taking the form $\pi(\xi_j)=K\xi_j^{-r}L(\xi_j)$, $\; r>1, \; K>0$. If $L(\xi)$ satisfies either of the two following conditions for all $\xi>0$: \\
$(C1)\;  1-C_{11}\xi^{-t} \leq L(\xi) \leq C_{12}, \quad C_{12}\geq1, \; C_{11}>0, \; t>0$;\\
$(C2)\;  C_{21}\xi^{-t}\leq L(\xi) \leq 1, \quad C_{21}>0,\; t>0$,\\
then \eqref{cond:highlynear0} and \eqref{cond:fattail} hold with $\tau_n \lesssim a_n^2p_n^{-(1+u')/(r-1)}$ for some $u'>0$ and $\log\tau_n \gtrsim -\log p_n$.
\end{theorem}

\begin{rmk}
It is easy to see that $-\log p_n \lesssim \log(a_n^2p_n^{-(1+u')/(r-1)})$, therefore such $\tau_n$ must exist.
\end{rmk}

\begin{rmk}
Many commonly used shrinkage priors satisfy either $(C1)$ or $(C2)$. As shown in Table \eqref{table: bounds for L}, mixing component $\pi(\xi)$ of student's t, TPBN (horseshoe, NEG are special cases of TPBN) and HIB satisfies $(C1)$; horseshoe+ satisfies $(C2)$. Proofs of these bounds are provided in Appendix.
\end{rmk}

\begin{rmk}
For application, we recommend using TPBN prior. It has been shown in \cite{bai2018high} that TPBN prior is easy for implementation using Gibbs sampling and relevant computation $R$ package MBSP is readily available. They compared their simulation results with other high-dimensional multivariate models.
\end{rmk}

\begin{table}[!htbp]
	\centering
	\caption{Bounds for $L(\xi)$}
	\label{table: bounds for L}
	\begin{tabular}{llll}	
		\hline\hline
		prior       & L($\xi$) & lower bound & upper bound   \\ \hline 
		Student's t &  $\exp(-a/\xi)$ & $1-a\xi^{-1}$ &1 \\ \hline
		TPBN        &  $(\xi/(1+\xi))^{a+u}$ & $1-(a+u)\xi^{-1}$ & 1  \\
		Horseshoe   &  $\xi/(1+\xi)$ & $1-\xi^{-1}$    & 1     \\
		NEG         &  $(\xi/(1+\xi))^{a+1}$ & $1-(a+1)\xi^{-1}$    & 1        \\ \hline
		GDP			& $\begin{aligned}
		&\int_{0}^{\infty} t^a\exp(-t-\eta\sqrt{2t/\xi})dt \\ 
		&\times 1 / \Gamma(a+1)\end{aligned}$ 
		& $1-\sqrt{2}\eta\frac{\Gamma(a+3/2)}{\Gamma(a+1)}x^{-1/2}$  & 1\\ \hline
		HIB			& $\begin{aligned}
		&\exp(-\frac{s}{1+\xi})(\phi^2+\frac{1-\phi^2}{1+\xi})^{-1}\\
		&\times (\frac{\xi}{1+\xi})^{a+u} (1\lor\phi^2)e^s
		\end{aligned}$
		&  $1-(a+u)\xi^{-1}$ & $(\phi^2\lor\frac{1}{\phi^2}) e^s$ \\ \hline
		Horseshoe+   & $\xi^{3/4}(\xi-1)^{-1}\log\xi / 4$  &     $\xi^{-1/4}/4$   & 1   \\ 
		\hline\hline
	\end{tabular}
\end{table}

\section{Numerical Experiments and Data Analysis}
Through numerical experiments, we examine the uncertainty assessment for covariance matrix estimate using scale mixture of Gaussians proposed in Section \ref{EMBSP}. We explore how the difference between estimation and truth varies as $n$ and $p_n$ grows. More simulations that evaluate performances on coefficient matrix reconstruction, prediction as well as variable selection under various situations are presented in \citet{bai2018high}. A real data analysis is also given.

\subsection{Numerical Experiments}
In our simulation, horseshoe mixing density $\pi(\xi)=\xi^{-1/2}(1+\xi)^{-1}$ is used. We focus on performance on ultra high-dimensional and ultra-sparse setting, where $p$ is approximately $n^{1.5}$, proportion of nonzero coefficients ranges from $0.38\%$ to $1.6\%$. Six different experiments settings are listed below.\\

\noindent Experiment 1: $n=25, \; p=125, \; d=3, \; s_0=2, \; s_0/p=1.6\%$.\\
Experiment 2: $n=50, \; p=354, \; d=3, \; s_0=3, \; s_0/p=0.85\%$.\\
Experiment 3: $n=75, \; p=650, \; d=3, \; s_0=4, \; s_0/p=0.62\%$.\\
Experiment 4: $n=100, \; p=1000, \; d=3, \; s_0=5, \; s_0/p=0.5\%$.\\
Experiment 5: $n=125, \; p=1398, \; d=3, \; s_0=6, \; s_0/p=0.43\%$.\\
Experiment 6: $n=150, \; p=1837, \; d=3, \; s_0=7, \; s_0/p=0.38\%$.\\

In all six experiments, data are generated according to the multivariate linear regression model \eqref{regmodel}. Each row of $\bX$ is generated independently from $\Nor_p(\mathbf{0}, \mathbf{\Gamma})$, where $\mathbf{\Gamma} _{ij} = 0.5^{|i-j|}$. The true coefficient matrix $\bBzero$ is generated by uniformly selecting $s_0$ nonzero rows , and other rows are set to be zero. For nonzero rows, each entry is independently sampled from $Unif([-5, -0.5]\cup[0.5, 5])$. The true covariance matrix $(\bSigma_0)_{ij}=\sigma^2(0.5)^{|i-j|}, \;  \sigma^2=2$. 

By Theorem \eqref{thm: scalemixureprior}, when Assumption (1)-(3) holds, the global shrinkage parameter $\tau_n$ for nearly minimax posterior contraction rate should satisfy $\tau_n \lesssim a_n^2p_n^{-(1+u')/(r-1)}$ for some $u'>0$ and $\log\tau_n \gtrsim -\log p_n$. But in application, this value is very small, e.g., such $\tau_n$ in Experiment 3 would be around $10^{-13}$. Too small $\tau_n$ will cause problems in Gibbs sampling\cite{van2014horseshoe}. Currently, inference for the global hyperparmeter is still an open problem\cite{piironen17onthehyperprior}. Here, we set $\tau_n=1/(p_n\sqrt{n\log n})$, which achieves posterior consistency, although theoretical posterior contraction rate is not available\cite{bai2018high}. We use the Gibbs sampler in $R$ package MBSP, where the major computational complexity is linear in $p_n$\cite{bhattacharya2016fast}. Each experiment is repeated 100 times. In all experiments, Gibbs sampler is run for 15000 iterations, the first 5000 iterations are burn-in. 

Posterior mean  
 $\hat{\bSigma}$ is taken to be the point estimators of 
  $\bSigma$. $\Twonorm{\hat{\bSigma}-\bSigma_0}$ and $\Fnorm{\hat{\bSigma}-\bSigma_0}$ are used to measure the difference between posterior estimates and the truth in two different norms. Figure \eqref{plot:ratio} illustrates how $\Twonorm{\hat{\bSigma}-\bSigma_0}$ and $\Fnorm{\hat{\bSigma}-\bSigma_0}$ decrease as $n$ and $p$ increase. Although this trend is a finite sample behavior, it matches our posterior consistency result established in previous section.

\begin{figure}[h]
	\centering
	\includegraphics[width = 5in, keepaspectratio]{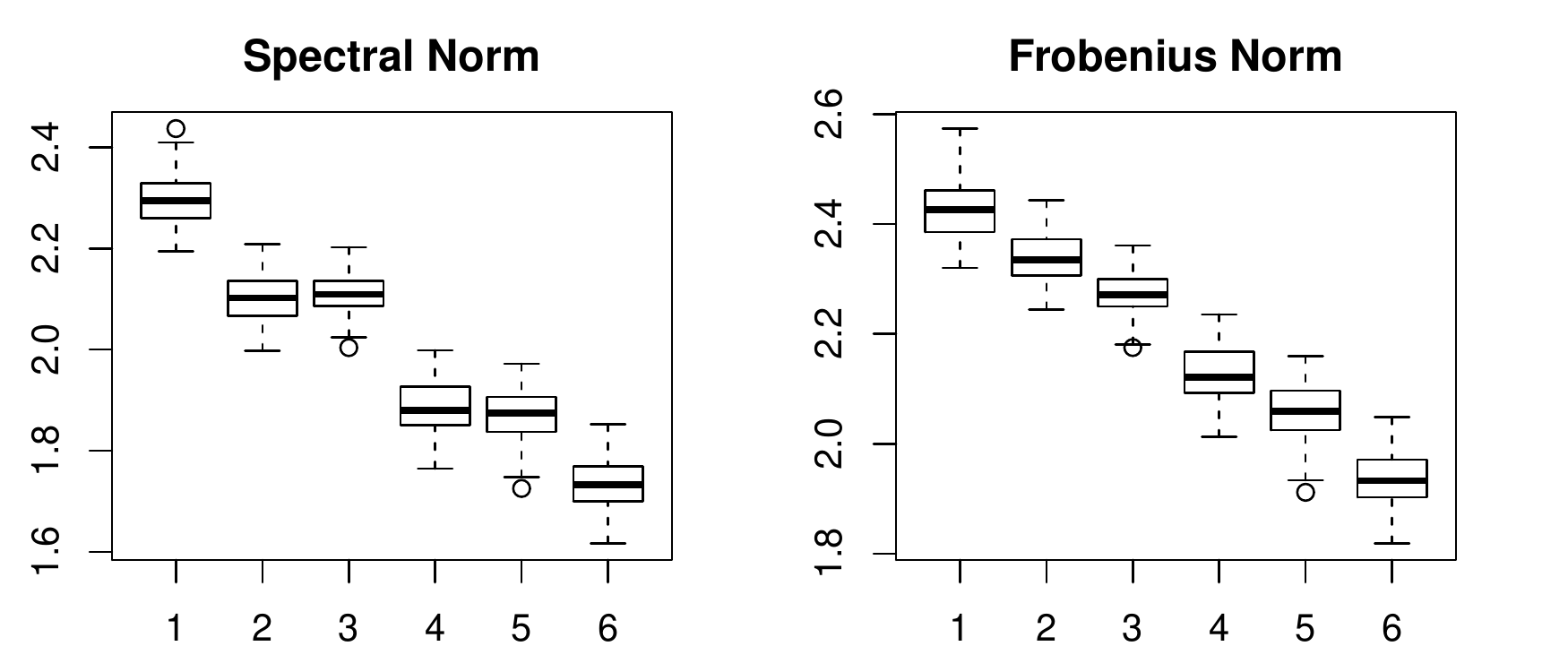}
		\caption{Box-plots of difference between estimated and true covariance matrix. The x-axis indicates experiment number and the y-axis indicates $\Twonorm{\hat{\bSigma}-\bSigma_0}$ (left) and 		$\Fnorm{\hat{\bSigma}-\bSigma_0}$ (right) respectively.}
		\label{plot:ratio}
\end{figure}

\subsection{Data Analysis}

We estimate the correlation between multiple responses on a yeast cell cycle data set. This data set was first analyzed by \citet{chun2010sparse} and is available in the spls package in R. 

In molecular biology, transcription factors (TFs), also known as sequence-specific DNA-binding factors, are proteins that controls the rate of transcription of genetic information from DNA to mRNA, by binding to a specific DNA sequence. To understand the regulatory mechanism of TFs, it is important to reveal the relationship between TFs and their target genes.

In the original yeast cell cycle data set, the response Y consists of 542 cell-cycle-regulated genes from an $\alpha$ factor arrest method and mRNA levels measured every 7 minutes at 18 time points, i.e. $n=542$, $q=18$. The $542\times106$ design matrix X consists of 106 TFs' binding information, representing the strength of interaction between TFs and the target genes. This data set has been analyzed in \citet{chen2012sparse, goh2017bayesian, bai2018high} with various variable selection and estimation methods for regression matrix $\bB$.

Here, we focus on the estimation of covariance $\bSigma$ between responses. Because no sparsity is assumed in $\bSigma$ or $\bSigma^{-1}$, $q=18$ is too large to get accurate estimation. We only use the first four measurements, i.e. $q=4$. The point estimator obtained for covariance matrix between responses is
\begin{equation*}
\hat{\bSigma} = 
\begin{bmatrix}
0.39 & 0.17 & 0.07 & -0.03\\
0.17 & 0.30 & 0.13 & 0.04\\
0.07 & 0.13 & 0.28 & 0.17\\
-0.03 & 0.04 & 0.17 & 0.25
\end{bmatrix}. 
\end{equation*}

Note that $Y_i$'s are mRNA levels measured every 7 minutes, it would be natural to observe autocorrelation as demonstrated in the above estimator.


\section{Conclusion and Future Work}
This paper has several contributions. First, we propose a set of general conditions for continuous prior $\pi(\bB|\bSigma)$ in sparse multivariate Bayesian estimation that can achieve nearly-optimal posterior contraction rate. While previous Bayesian multivariate models usually assume $\bSigma$ to be fixed and known, our work highlights the proof of posterior contraction of both coefficient matrix $\bB$ and covariance matrix $\bSigma$. Moreover, we allow $p_n$ to grow nearly at an exponential rate with $n$ and response dimension $q_n$ to go to infinity. To the best of our knowledge, our work is the first paper showing the nearly-optimal contraction rate of continuous shrinkage priors under this setting. The tools we developed in proof can also be utilized in other multivariate Bayesian models. For application, we show that a large family of heavy-tailed priors, including Student'’s t prior, horseshoe and horseshoe+ prior, the generalized double Pareto prior, etc, satisfy the condition with good posterior contraction results.
 
Although we have established an informative $\ell_2$ reconstruction rate, there are still many important issues unexplored. One of them is the sparsity of $\bSigma$ or its inverse. In our paper, where no structure of $\bSigma$ is assumed, although dimension of response $q_n$ is allowed to grow, it has to be much smaller than sample size $n$ in order to keep $\bSigma$ consistently estimable\cite{ning2018bayesian}. Recently, to encourage sparsity of precision matrix, \citet{li2019joint} proposed a model putting horseshoe prior on regression coefficient and graphical horseshoe prior on precision matrix.  

Another interesting problem is whether to adopt the joint scale-invariant prior framework. We use a scale-invariant prior in the paper, but this may result in underestimating the model error\cite{moran1801variance}. \citet{moran1801variance} recommend independent priors for regression coefficient and error variance \textit{apriori} for preventing distortion of the global-local shrinkage mechanism and obtaining better estimates of the error variance.  

\section*{Acknowledgments}
The authors are grateful to Ray Bai and Qian Qin for helpful comments and suggestions.

\bibliographystyle{elsarticle-harv}
\bibliography{multi_posterior_ref}

\section*{Appendix}
	
	\begin{proof} (Theorem \ref{thm:posterior contraction})
		Define auxiliary sets $A_n$, $B_n$ and $C_n$ are constructed as follow.  
		\begin{align*}
		A_n = &\{\mbox{at least } \tilde{p} \mbox{ entries } \norm{\bB_j\bSigma^{-1/2}} \mbox{ is larger than } a_n\} \\
		\cup  &\{\norm{\bSigma-\bSigma_0} \geq M(\norm\bSigma \lor \norm{\bSigma_0})\epsilon_n \} \\
		\cup  &\{\Fnorm{(\bB-\bBzero)\bSigma_0^{-1/2}} \geq M \epsilon_n \},
		\end{align*}
		
		$B_n = \{\mbox{at least } \tilde{p} \mbox{ entries }  \norm{\bB_j\bSigma^{-1/2}} \mbox{ is larger than } a_n\}$,  
		and $C_n=A_n \setminus B_n$. Let $\theta = (\bB,\bSigma)$ and $\theta_0 = (\bBzero,\bSigma_0)$. By Lemma A.4 in \citet{song2017nearly}, it suffices to show the following three parts:
		\begin{equation} \label{prior}
		\pi(B_n) \leq e^{-k_1 n\epsilon_n^2} 
		\end{equation}
		
		There exists a test function $\phi_n$ s.t.
		\begin{align} 
		\E_{\theta_0} \phi_n &\leq e^{-k_2 n\epsilon_n^2}, \label{test1}\\
		\sup_{\theta \in C_n} \E_{\theta} (1-\phi_n) &\leq e^{-k_3 n\epsilon_n^2} \label{test2}
		\end{align}

		And for sufficiently large $n$,
		\begin{equation} \label{marginal}
		P_{\theta_0}\big( \dfrac{m(\Data)}{f_{\theta_0}(\Data)} \geq e^{-k_4 n\epsilon_n^2}\big) \geq 1-e^{-k_5 n\epsilon_n^2}
		\end{equation}
		for some constant $0<k_4<\min(k_1,k_3)$,
		where $m(\Data)=\int_{\Theta} \pi(\theta)f_{\theta}(\Data) d\theta$ is the marginal of $\Data$.
		
		So the proof is composed of three parts: 
		(\RomanNumeralCaps 1)construction of test $\phi_n$ satisfying \eqref{test1} and \eqref{test2}, 
		(\RomanNumeralCaps 2) showing that event $B_n$ has very probability under the specified prior,
		(\RomanNumeralCaps 3) showing that the marginal probability of data is highly likely to be bounded away from 0 if data is generated with true parameters.\\
		
		\textbf{Part \RomanNumeralCaps 1}: Firstly, we show \eqref{test1} and \eqref{test2} by constructing testing function $\phi_n$ in the following way. \\
		For given $S\subseteq \{1,\cdots, p\}$, consider the following testing functions $\phi_{n,S}^{(1)}$ and $\phi_{n,S}^{(2)}$:
		\begin{align*}
		\phi_{n,S}^{(1)} & = 1\{
		\norm{\frac{1}{n-\det{S}}\bY^T(I_n-H_S)\bY - \bSigma_0}\geq M \norm\bSigma_0\epsilon_n/2 \} \\
		\phi_{n,S}^{(2)} & = 1\{\Fnorm{((\bX_S^T \bX_S)^{-1} \bX_S^T\bY - \bBzero_S)\sqrtbSigmaOInv} \geq M\epsilon_n/2\}
		\end{align*}
		where $\bX_S$ is the submatrix of $\bX$ composed of columns indexed by $S$, $\bB_S$ and $\bBzero_S$ are the submatrices of $\bB$ and $\bBzero$ composed of rows indexed by  $S$ respectively, $H_S = \bX_S(\bX_S^T\bX_S)^{-1} \bX_S^T$.
		
		We have the following two inequalities for $\E_{(\bBzero,\bSigma_0)}\phi_{n,S}^{(1)}$ and $\E_{(\bBzero,\bSigma_0)}\phi_{n,S}^{(2)}$.
		\begin{equation*}
		\begin{split}
		\E_{(\bBzero,\bSigma_0)}\phi_{n,S}^{(1)} & = 
		P_{(\bBzero,\bSigma_0)}
		(\norm{\frac{1}{n-\det{S}}\bY^T(I_n-H_S)\bY - \bSigma_0}\geq M \norm\bSigma_0\epsilon_n/2) \\
		& \leq e^{-c M^2 n\epsilon_n^2/(4K^2)} \quad \mbox{by Lemma \eqref{Lemma1}}
		\end{split}
		\end{equation*}
		
		\begin{equation*}
		\begin{split}
		\E_{(\bBzero,\bSigma_0)}\phi_{n,S}^{(2)} & = P(\Fnorm{(\bX_S^T \bX_S)^{-1} \bX_S^T \bepsilon} \geq M\epsilon_n/2)\\
		&\leq P(\lambda_{\max}(\bX_S^T \bX_S)^{-1} tr(\bepsilon^T H_S \bepsilon) \geq M^2 \epsilon_n^2/4)\\
		&\leq P(\chi^2_{q_n\det{S}} \geq M^2\lambda_0 n\epsilon_n^2/4) \\
		&\leq e^{-\lambda_0 M^2 n\epsilon_n^2/16}
		\end{split}
		\end{equation*}
		The last inequality holds because $q_n\lesssim \log p_n$, $\det{S} \simeq s_0$ and $P(\chi^2_p \geq x) \leq e^{-x/4}$ if $x\geq 8p$\cite{armagan2013posterior}.
		
		Let $\phi_n = \max\{\phi_n^{(1)}, \phi_n^{(2)}\}$, where 
		\[\phi_n^{(i)} = \max_{\{S\supset S_0, |S|\leq\tilde{p}+s_0\}} \phi_{n,S}^{(i)}, \quad i=1,2. \]
		
		\begin{equation*}
		\begin{split}
		\E_{(\bBzero,\bSigma_0)} \phi_n 
		& \leq \E_{(\bBzero,\bSigma_0)} \sum_{\{S\supset S_0, |S|\leq\tilde{p}+s_0\}}(\phi_{n,S}^{(1)} + \phi_{n,S}^{(2)})\\
		& \leq \sum_{i=0}^{\tilde{p}}  {p_n-s_0 \choose i} 2 e^{-k_{23} Mn\epsilon_n^2} 
		\quad \mbox{ where }k_{23}=\min(\frac{cM}{4K^2},\frac{\lambda_0M}{16})\\
		& \leq 2 (\tilde{p}+1)(p_n-s_0)^{\tilde{p}} e^{-k_{23} Mn\epsilon_n^2}
		\end{split}
		\end{equation*}
		
		Taking logarithm on both sides, 
		\begin{equation*}
		\begin{split}
		\log(\E_{(\bBzero,\bSigma_0)} \phi_n) 
		&\leq \tilde{p} \log(p_n-s_0) + \log(2(\tilde{p}+1)) - k_{23}Mn\epsilon_n^2 \\
		&\leq (\tilde{p}+1) \log p_n - k_{23}Mn\epsilon_n^2 \\
		&\leq - k_{23}Mn\epsilon_n^2 /2 \quad \mbox{ where } 
		\tilde{p} = \lfloor \dfrac{k_{23}M n\epsilon_n^2 }{2\log p_n} \rfloor -1
		\end{split}
		\end{equation*}
		
		Note that sufficient large $M (> 6/k_{23})$ will ensure there is such $\tilde{p}(\geq 2)$ such that 
		$\E_{(\bBzero,\bSigma_0)} \phi_n \leq e^{-k_2 n\epsilon_n^2}$.\\
		
		Now we want to show $\sup_{(\bB,\bSigma) \in C_n} \E_{(\bB,\bSigma)} (1-\phi_n) \leq e^{-k_3 n\epsilon_n^2}$. Consider the following two sets $C_{n_1}$ and $C_{n,2}$:
		\begin{equation*}
		\begin{split}
		C_{n,1} = &\{\norm{\bSigma-\bSigma_0} \geq M(\norm\bSigma \lor \norm\bSigma_0)\epsilon_n \} \cap
		\{\mbox{at most } \tilde{p} \mbox{ entries } \norm{\bB_j\bSigma^{-1/2}} \mbox{ is larger than } a_n\},\\
		C_{n,2} = &\{\Fnorm{(\bB-\bBzero)\bSigma_0^{-1/2}} \geq M \epsilon_n, \norm{\bSigma-\bSigma_0} \leq M(\norm\bSigma \lor \norm\bSigma_0)\epsilon_n \} \cap\\ 
		&\{\mbox{at most } \tilde{p} \mbox{ entries } \norm{\bB_j\bSigma^{-1/2}} \mbox{ is larger than } a_n\}.
		\end{split}
		\end{equation*}
		
		It's easy to verify that $C_n \subset C_{n,1}\cup C_{n,2}$, so we have
		\begin{equation*}
		\begin{split}
		\sup_{(\bB,\bSigma) \in C_n} \E_{(\bB,\bSigma)} (1-\phi_n) &= 
		\sup_{(\bB,\bSigma) \in C_n} \E_{(\bB,\bSigma)} \min(1-\phi_n^{(1)}, 1-\phi_n^{(2)}) \\
		&\leq \max \{
		\sup_{(\bB,\bSigma) \in C_{n,1}} \E_{(\bB,\bSigma)} (1-\phi_n^{(1)}),
		\sup_{(\bB,\bSigma) \in C_{n,2}} \E_{(\bB,\bSigma)} (1-\phi_n^{(2)})
		\} 
		\end{split}
		\end{equation*}
		
		By definition of $\phi_n^{(1)}$, for $\forall S\supset S_0, \det{S}\leq \tilde{p}+s_0$, we have
		\[
		\sup_{
			(\bB,\bSigma) \in C_{n,1}} \E_{(\bB,\bSigma)} (1-\phi_n^{(1)}) \leq 
		\sup_{(\bB,\bSigma) \in C_{n,1}} \E_{(\bB,\bSigma)} (1-\phi_{n,S}^{(1)}).
		\]
		
		Taking $S = \{j: \Twonorm{\bB_j\sqrtbSigmaInv} \geq a_n \} \cup S_0$, note that when $(\bB,\bSigma) \in C_{n,1}$, $\norm{\bSigma-\bSigma_0} - M\norm\bSigma_0\epsilon_n/2\geq M\norm\bSigma\epsilon_n/2$, then
		\begin{equation*}
		\begin{split}
		&\sup_{(\bB,\bSigma) \in C_{n,1}} \E_{(\bB,\bSigma)} (1-\phi_{n,S}^{(1)}) \\
		=& \sup_{(\bB,\bSigma) \in C_{n,1}} P(\norm{\frac{1}{n-\det{S}} \bY^T(I_n-H_S)\bY-\bSigma_0}
		\leq M\norm\bSigma_0\epsilon_n/2 ) \\
		\leq &\sup_{(\bB,\bSigma) \in C_{n,1}} P(\norm{\frac{1}{n-\det{S}} \bY^T(I_n-H_S)\bY-\bSigma}
		\geq \norm{\bSigma-\bSigma_0} - M\norm\bSigma_0\epsilon_n/2 ) \\
		\leq & \sup_{(\bB,\bSigma) \in C_{n,1}} P(\norm{\frac{1}{n-\det{S}} \bY^T(I_n-H_S)\bY - \bSigma} \geq M\norm\bSigma\epsilon_n/2 ) \\
		\leq& e^{-c M^2 n\epsilon_n^2/(16K^2)}
		\end{split}
		\end{equation*}
		The last inequality follows by Lemma \eqref{Lemma1} because $\bY= \bX\bB + \bepsilon\sqrtbSigma$ and for $(\bB,\bSigma) \in C_{n,1}$,
		\begin{align*}
		\Fnorm{(I_n-H_S) X\bB\sqrtbSigmaInv}^2 &= \Fnorm{X_{S^c}\bB_{S^c}\sqrtbSigmaInv}^2 \\
		&\leq \Fnorm{X_{S^c}}^2 \Fnorm{\bB_{S^c}\sqrtbSigmaInv}^2\\
		&\leq (n p_n) (p_n a_n^2) \\
		&\lesssim n\epsilon_n^2 .
		\end{align*}
		Now we consider $\sup_{(\bB,\bSigma) \in C_{n,2}} \E_{(\bB,\bSigma)} (1-\phi_{n,S}^{(2)})$. For $(\bB,\bSigma) \in C_{n,2}$, $\norm{\bSigma}-\norm{\bSigma_0} \leq M\norm\bSigma\epsilon_n$ or $\norm{\bSigma}-\norm{\bSigma_0} \leq M\norm{\bSigma_0}\epsilon_n$, so $\norm\bSigma \leq 2\norm{\bSigma_0}$, which gives $\norm{\bSigma\InvbSigma_0} \leq 2b_2/b_1$ by assumption $A_3(3)$. So we have 
		\[
		\Fnorm{(\bB_S-\bBzero_S)\sqrtbSigmaOInv} \geq \Fnorm{(\bB-\bBzero)\bSigma_0^{-1/2}} 
		- \Fnorm{\bB_{S^c}\sqrtbSigmaOInv} 
		\geq 7M\epsilon_n/8
		\]
		since $\Fnorm{\bB_{S^c}\sqrtbSigmaOInv} \leq \Fnorm{\bB_{S^c}\sqrtbSigmaInv}\norm{\bSigma\InvbSigma_0}^{1/2}$, $\Fnorm{\bB_{S^c}\sqrtbSigmaInv}\leq a_n p_n \prec \epsilon_n$ and $\norm{\bSigma\InvbSigma_0}^{1/2}$ is bounded.
		In addition, we have
		\begin{equation*}
		\begin{split}
		\Fnorm{(\bX_S^T \bX_S)^{-1} \bX_S^T X_{S^c}\bB_{S^c}\sqrtbSigmaOInv} &\leq
		\sqrt{\lambda_{\max}((\bX_S^T \bX_S)^{-1})} \Fnorm{X_{S^c}\bB_{S^c}\sqrtbSigmaOInv}\\
		& \leq\sqrt{1/(n\lambda_0)}	\norm{\bSigma\InvbSigma_0}^{1/2}
		\Fnorm{X_{S^c}\bB_{S^c}\sqrtbSigmaInv}\\
		& \leq 
		M\epsilon_n/8 \mbox{ for } M\geq 8\sqrt{b_2/(b_1\lambda_0)} 
		\end{split}
		\end{equation*}
		
		Therefore, 
		\begin{equation*}
		\begin{split}
		&\sup_{(\bB,\bSigma) \in C_{n,2}} \E_{(\bB,\bSigma)} (1-\phi_{n,S}^{(2)}) \\
		= & \sup_{(\bB,\bSigma) \in C_{n,2}}
		P(\Fnorm{((\bX_S^T \bX_S)^{-1} \bX_S^T \bY - \bBzero_S)\sqrtbSigmaOInv} \leq M\epsilon_n/2) \\
		=& 
		\begin{multlined}
		\sup_{(\bB,\bSigma) \in C_{n,2}}
		P(\Fnorm{ (\bB_S-\bBzero_S)\sqrtbSigmaOInv + 
			(\bX_S^T \bX_S)^{-1} \bX_S^T X_{S^c}\bB_{S^c}\sqrtbSigmaOInv  \\
			+ (\bX_S^T \bX_S)^{-1} \bX_S^T E \sqrtbSigma\sqrtbSigmaOInv } 
		\leq M\epsilon_n/2)
		\end{multlined}\\
		\leq& 
		\begin{multlined}
		\sup_{(\bB,\bSigma) \in C_{n,2}}
		P(\Fnorm{(\bB_S-\bBzero_S)\sqrtbSigmaOInv} -
		\Fnorm{(\bX_S^T \bX_S)^{-1} \bX_S^T X_{S^c}\bB_{S^c}\sqrtbSigmaOInv} \\
		- \Fnorm{(\bX_S^T \bX_S)^{-1} \bX_S^T E \sqrtbSigma\sqrtbSigmaOInv }
		\leq M\epsilon_n/2)
		\end{multlined} \\
		\leq&  \sup_{(\bB,\bSigma) \in C_{n,2}}
		P(\Fnorm{(\bX_S^T \bX_S)^{-1} \bX_S^T E \sqrtbSigma\sqrtbSigmaOInv }
		\geq M\epsilon_n/4) \\
		\leq & P(\chi^2_{q_n\det{S}} \geq \lambda_0 b_1 M^2 n\epsilon_n^2/(32b_2)) \\
		\leq& e^{-\lambda_0 b_1 M^2 n\epsilon_n^2/(128 b_2)}
		\end{split}
		\end{equation*}
		
		The second last inequality holds since 
		\begin{equation*}
		\begin{split}
		\Fnorm{(\bX_S^T \bX_S)^{-1} \bX_S^T E \sqrtbSigma\sqrtbSigmaOInv }^2
		&\leq \Fnorm{(\bX_S^T \bX_S)^{-1} \bX_S^T E}^2 \norm{\bSigma\InvbSigma_0} \\
		&\leq 2(b_2/b_1) \lambda_{\max} ((\bX_S^T \bX_S)^{-1})
		\Fnorm{H_S E}^2 \\
		&\leq \frac{2b_2}{b_1 \lambda_0 n} \chi^2_{q_n\det{S}}
		\end{split}
		\end{equation*}
		
		Let $k_3 = \min(\frac{cM^2}{16K^2}, \frac{\lambda_0 b_1 M^2}{128 b_2})$, \eqref{test2} is proved.\\
		
		\textbf{Part \RomanNumeralCaps 2}: Now we will show $\pi(B_n) \leq e^{-k_1 n\epsilon_n^2}$. \\
		Define $N = \det{\{j:\Twonorm{\bB_j\sqrtbSigmaInv} \geq a_n\}}$, following the proof in Part II of Theorem A.1 in \citet{song2017nearly}, we have $\pi(B_n)\leq e^{-t_n}/(2\sqrt{\pi t_n})$, where
		\begin{equation*}
		\begin{split}
		t_n &= (\tilde{p}-1)\log(\tilde{p}-1) + (\tilde{p}-1)\log\dfrac{1}{p_n v_n} + 
		(p_n - \tilde{p} + 1)\log\dfrac{p_n-\tilde{p} + 1}{p_n - p_n v_n}\\
		&\geq \frac{u}{2} \tilde{p}\log p_n - (p_n - \tilde{p} + 1)\log(1 + \dfrac{\tilde{p} - 1 -p_n v_n}{p_n - \tilde{p} + 1}) \quad
		(\mbox{recall that } p_nv_n \leq p_n^{-u})
		\\
		&\geq \frac{u}{2} \tilde{p}\log p_n - (\tilde{p} - 1 -p_n v_n)\\
		&\geq \frac{u}{4} \tilde{p}\log p_n\\
		&\geq uk_{23}M n\epsilon_n^2/16
		\end{split}
		\end{equation*}
		
		Thus $\pi(B_n) \leq e^{-k_1 n\epsilon_n^2}$ is proved with $k_1 = uk_{23}M/16$.\\

		\textbf{Part \RomanNumeralCaps 3}: At last, we will show \eqref{marginal}. \\
		By Part III in \citet{song2017nearly}, it suffices to show that 
		\begin{equation} \label{marginalsufficient}
		\begin{split}
		P_{\theta_0}( &
		\pi(\Fnorm{(\bY-X\bB)\sqrtbSigmaInv}^2 - \Fnorm{(\bY-X\bBzero)\sqrtbSigmaOInv}^2 + n\log(\dfrac{\det\bSigma}{\det{\bSigma_0}}) \leq k_4 n\epsilon_n^2/2) \\
		&\geq  e^{-k_4 n\epsilon_n^2/2}) \geq 1-e^{-k_5 n\epsilon_n^2}
		\end{split}
		\end{equation}
		
		Proof of \eqref{marginalsufficient} has three steps. 
		The first step is to show event $\Omega \coloneqq \{\FnormSq{\bepsilon} \leq n q_n (1+c_1) \mbox{ and } \TwoInfnorm{\bX^T \bepsilon} \leq c_2 n\epsilon_n \}$ has large probability closing to 1. In second step, we show that on event $\Omega$, and when the data $\bY$ is generated under true parameter $(\bBzero,\bSigma_0)$,
		$ \{
		\FnormSq{(\bY-X\bB)\sqrtbSigmaInv} - \FnormSq{(\bY-X\bBzero)\sqrtbSigmaOInv} + n\log(\det\bSigma / \det{\bSigma_0}) \leq k_4 n\epsilon_n^2/2)
		\}$
		is a super-set of 
		$\{\cap_{i=1}^{d}\{\tilde{\bSigma}: 1\leq \lambda_i(\tilde{\bSigma})\leq 1+\epsilon_n^2/q_n\} \mbox{ and }
		\TwoOnenorm{(\bB - \bBzero)\sqrtbSigmaInv} \leq \eta\epsilon_n \}$,
		where $\tilde{\bSigma} \coloneqq \bSigma_0^{-1/2}\bSigma\bSigma_0^{-1/2}$, and $\eta$ is a constant satisfies $\eta\epsilon_n/p_n \geq a_n$. 
		The last step is to get the lower bound of $\pi(\{\cap_{i=1}^{d}\{\tilde{\bSigma}: 1\leq \lambda_i(\tilde{\bSigma})\leq 1+\epsilon_n^2/q_n\} \mbox{ and }
		\TwoOnenorm{(\bB - \bBzero)\sqrtbSigmaInv} \leq \eta\epsilon_n \})$.\\
		
		\textbf{Step 1}: $P(\{\FnormSq{\bepsilon} \leq n q_n (1+c_1) \mbox{ and } \TwoInfnorm{\bX^T \bepsilon} \leq c_2 n\epsilon_n \}) \geq 1-e^{-n \epsilon_n^2}$ \\
		It's easy to see that $P(\FnormSq{\bepsilon} > n q_n (1+c_1)) \leq e^{-c_1^2n q_n/4} \leq e^{-c_1^2n \epsilon_n^2/4} $ because $\FnormSq{\bepsilon} \sim \chi^2_{nq_n}$. And $(\bX^T \bepsilon)_{ij}^2/n \leq (\bX^T \bepsilon)_{ij}^2/(\bX^T \bX)_{ii}$ by the fact that $\bX$ is uniformly bounded by 1. And note that for fixed $i$, $(\bX^T \bepsilon)_{ij}/\sqrt{(\bX^T\bX)_{ii}} \sim \Nor(0,1) \; independently, \; j=1,\cdots,q_n$. 
		Thus we have 
		\begin{equation*}
		\begin{split}
		P(\TwoInfnorm{\bX^T \bepsilon} > c_2 n\epsilon_n) 
		&= P(\max_{1\leq i \leq n} \frac{1}{n}\sum_{j=1}^{q_n} (\bX^T \bepsilon)_{ij}^2 >  c_2^2 n\epsilon_n^2) \\
		&\leq P(\max_{1\leq i \leq n} Z_i >  c_2^2 n\epsilon_n^2)  
		\quad \mbox{ where } Z_i\sim \chi^2_{q_n}\\
		&\leq p_n P(\chi^2_{q_n} > c_2^2 n\epsilon_n^2) \\
		&\leq p_n \exp\{-c_2^2 n\epsilon_n^2/4\} \\
		&\leq \exp\{-(c_2^2/4-1/s_0) n\epsilon_n^2\} \\
		&\leq e^{-3 n\epsilon_n^2}
		\end{split}
		\end{equation*}
		
		$P(\Omega) \geq 1- P(\FnormSq{\bepsilon} > nq_n(1+c_1)) - P(\TwoInfnorm{X^T E} > c_2 n\epsilon_n) \geq 1 - e^{-n\epsilon_n^2}$. \\
		
		\textbf{Step 2}: On event $\Omega$ and when the data $\bY$ is generated under true parameter $(\bBzero,\bSigma_0)$, if 
		$(\bB,\bSigma) \in \{\cap_{i=1}^{d}\{\tilde{\bSigma}: 1\leq \lambda_i(\tilde{\bSigma})\leq 1+\epsilon_n^2/q_n\} \mbox{ and }
		\TwoOnenorm{(\bB - \bBzero)\sqrtbSigmaInv} \leq \eta\epsilon_n \}$,
		then the following four inequalities hold.
		\begin{align}
		n\log(\det\bSigma / \det{\bSigma_0}) &\leq n\epsilon_n^2 \label{step2.1}\\
		\det{\FnormSq{\bepsilon\sqrtbSigmaO\sqrtbSigmaInv} - \FnormSq{\bepsilon}} &\leq (1+c_1)n\epsilon_n^2 \label{step2.2} \\
		\FnormSq{\bX(\bB - \bBzero)\sqrtbSigmaInv} &\leq \eta^2n\epsilon_n^2 \label{step2.3} \\
		\det{tr(\bX(\bB-\bBzero)\InvbSigma \sqrtbSigmaO E)} &\leq c_2\eta n\epsilon_n^2 \label{step2.4}
		\end{align}

		\eqref{step2.1} holds because 
		\[
		n\log(\det\bSigma / \det{\bSigma_0}) = n\sum_{i=1}^{q_n}\log\lambda_i(\tilde{\bSigma})
		\leq n\sum_{i=1}^{q_n}\log(1+\epsilon_n^2/q_n) \leq n\epsilon_n^2.
		\]
		
		For \eqref{step2.2}, it's easy to see
		\[
		\FnormSq{\bepsilon}\lambda_1(\tilde{\bSigma}^{-1} - I_{q_n}) \leq tr(\bepsilon^T \bepsilon (\tilde{\bSigma}^{-1}  - I_{q_n})) \leq 
		\FnormSq{\bepsilon}\lambda_{q_n}(\tilde{\bSigma}^{-1} - I_{q_n}).
		\]
		And
		$ 1-\epsilon_n^2 \leq \lambda_1(\tilde{\bSigma}^{-1}) \leq \lambda_{q_n}(\tilde{\bSigma}^{-1}) \leq 1 $ because
		$ 1\leq \lambda_i(\tilde{\bSigma})\leq 1+\epsilon_n^2/q_n \Rightarrow 
		1-\epsilon_n^2/q_n\leq \lambda_{q_n-i+1}(\tilde{\bSigma})\leq 1 $. So we have
		\[
		\det{\FnormSq{\bepsilon\sqrtbSigmaO\sqrtbSigmaInv} - \FnormSq{\bepsilon}} = 
		\det{tr(\bepsilon^T \bepsilon (\tilde{\bSigma}^{-1} - I_{q_n}))} \leq 
		\FnormSq{\bepsilon} \epsilon_n^2/q_n \leq (1+c_1) n\epsilon_n^2.
		\]
		For \eqref{step2.3}, let $\bX_i$ denote the $ith$ row of $\bX$, we have 
		\begin{equation*}
		\begin{split}
		\Twonorm{\bX_i (\bB - \bBzero)\sqrtbSigmaInv} &=
		\Twonorm{\sum_{j=1}^{p_n} x_{ij} (\bB - \bBzero)_j \sqrtbSigmaInv} \\
		&\leq \sum_{j=1}^{p_n} \det{x_{ij}} \Twonorm{(\bB - \bBzero)_j \sqrtbSigmaInv} \\
		&\leq \TwoOnenorm{(\bB - \bBzero)\sqrtbSigmaInv} 
		\end{split}
		\end{equation*}
		So $\FnormSq{\bX(\bB - \bBzero)\sqrtbSigmaInv} = \sum_{i=1}^n\Twonorm{\bX_i (\bB - \bBzero)\sqrtbSigmaInv}^2 \leq \eta^2n\epsilon_n^2 $.\\
		
		\eqref{step2.4} is immediately by Lemma \eqref{Lemma2}
		\begin{equation*}
		\begin{split}
		\det{tr(\bX(\bB-\bBzero)\InvbSigma \sqrtbSigmaO \bepsilon)} &= \det{tr(\bepsilon^T \bX(\bB-\bBzero)\InvbSigma \sqrtbSigmaO)} \\
		&\leq \TwoInfnorm{\bX^T \bepsilon} \TwoOnenorm{(\bB-\bBzero)\InvbSigma \sqrtbSigmaO}\\
		&\leq \TwoInfnorm{\bX^T \bepsilon} \lambda_{q_n}(\tilde{\bSigma}^{-1})^{1/2}
		\TwoOnenorm{(\bB-\bBzero)\sqrtbSigmaInv}\\
		&\leq c_2\eta n\epsilon_n^2
		\end{split}
		\end{equation*}
		
		Combining \eqref{step2.1}-\eqref{step2.4}, we can get
		\begin{equation*}
		\begin{split}
		& \FnormSq{(\bY-\bX\bB)\sqrtbSigmaInv} - \FnormSq{(\bY-\bX\bBzero)\sqrtbSigmaOInv} + n\log(\det\bSigma / \det{\bSigma_0})\\
		=& \FnormSq{\bX(\bB - \bBzero)\sqrtbSigmaInv} + 2tr(X(\bB-\bBzero)\InvbSigma \sqrtbSigmaO E) + \FnormSq{\bepsilon\sqrtbSigmaO\sqrtbSigmaInv} - \FnormSq{\bepsilon} + n\log(\det\bSigma / \det{\bSigma_0}) \\
		\leq& k_4 n\epsilon_n^2 \quad \mbox{when }k_4\geq 2+c_1+\eta^2+2c_2\eta
		\end{split}
		\end{equation*}
		
		\textbf{Step 3}: We want to show 
		\[
		\pi(\cap_{i=1}^{q_n}\{\tilde{\bSigma}: 1\leq \lambda_i(\tilde{\bSigma})\leq 1+\epsilon_n^2/q_n\} \mbox{ and }
		\TwoOnenorm{(\bB - \bBzero)\sqrtbSigmaInv} \leq \eta\epsilon_n) 
		\geq e^{-k_4 n \epsilon_n^2/2} .
		\]
		
		By Lemma A.4 in \citet{ning2018bayesian}, 
		\begin{equation} \label{step3.1}
		-\log(\pi(\cap_{i=1}^{q_n}\{\tilde{\bSigma}: 1\leq \lambda_i(\tilde{\bSigma})\leq 1+\epsilon_n^2/q_n\}) \lesssim q_n^2\log q_n + q_n^2\log\frac{1}{\epsilon_n^2} \lesssim q_n^2 \log n
		\lesssim n \epsilon_n^2
		\end{equation}
		
		Now, we look at 
		\[
		\pi(\TwoOnenorm{(\bB - \bBzero)\sqrtbSigmaInv} \leq \epsilon_n | \cap_{i=1}^{q_n}\{\tilde{\bSigma}: 1\leq \lambda_i(\tilde{\bSigma})\leq 1+\epsilon_n^2/q_n\}).
		\]
		It is easy to see that
		\begin{equation*}
		\begin{split}
		\{\TwoOnenorm{(\bB - \bBzero)\sqrtbSigmaInv} \leq \eta\epsilon_n \} \supset &
		\{\Twonorm{\bB_j\sqrtbSigmaInv} \leq \eta\epsilon_n/p_n \mbox{ for all } j\notin S\}  \cap\\
		& 
		\{\Twonorm{(\bB_j - \bBzero_j)\sqrtbSigmaInv} \leq \eta\epsilon_n/s_0  \mbox{ for all } j\in S_0\} 
		\end{split}
		\end{equation*}
		
		Firstly,
		\begin{equation} \label{step3.2}
		\begin{split}
		\pi(\Twonorm{\bB_j\sqrtbSigmaInv} \leq \eta\epsilon_n/p_n \mbox{ for } j\notin S_0) 
		&\geq \prod_{j\notin S_0}\pi(\Twonorm{\bB_j\sqrtbSigmaInv} \leq a_n) \\
		&\geq (1-p_n^{-(1+u)})^{p_n} \to 1
		\end{split}
		\end{equation}
		
		For a $d$-dimensional vector $\bx_0$, let $\Ball(\bx_0, r)$ denotes a ball in $\mathbb{R}^d$ with center $\bx_0$ and radius $r$. Note that conditioning on $\cap_{i=1}^{q_n}\{\tilde{\bSigma}: 1\leq \lambda_i(\tilde{\bSigma})\leq 1+\epsilon_n^2/q_n\}$, 
		\[
		\Ball(\bBzero_j\sqrtbSigmaInv, \eta\epsilon_n/s_0)\subset 
		\Ball(\boldsymbol{0}, \Twonorm{\bBzero_j\sqrtbSigmaInv} + \eta\epsilon_n/s_0) \subset
		\Ball(\boldsymbol{0}, \gamma \TwoInfnorm{\bBzero\sqrtbSigmaOInv})
		\]
		because 
		$\Twonorm{\bBzero_j\sqrtbSigmaInv} \leq \Twonorm{\bBzero_j\sqrtbSigmaOInv} \sqrt{\lambda_{q_n}(\tilde{\bSigma}^{-1})} \leq \Twonorm{\bBzero_j\sqrtbSigmaOInv}$. 
		
		Therefore, we have
		\begin{equation*}
		\begin{split}
		& \pi(\Twonorm{(\bB_j - \bBzero_j)\sqrtbSigmaInv} \leq \eta\epsilon_n/s_0  \mbox{ for all } j\in S_0 \vert
		\cap_{i=1}^{d}\{\tilde{\bSigma}: 1\leq \lambda_i(\tilde{\bSigma})\leq 1+\epsilon_n^2/q_n\})) \\
		\geq& \big(
		\mbox{Volumn}(\Ball(\bBzero_j\sqrtbSigmaInv, \eta\epsilon_n/s_0))
		\inf_{\Twonorm{\bB_j\sqrtbSigmaInv} \leq M_0}
		g_\tau(\bB_j\sqrtbSigmaInv)
		\big)^{s_0} \\
		=& \big(\dfrac{\pi^{q_n/2}}{\Gamma(q_n/2+1)}\big)^{s_0}
		\big(\dfrac{\eta\epsilon_n}{s_0}\big)^{q_n s_0}
		\big(\inf_{\Twonorm\bx \leq M_0}
		g_\tau(\bx)
		\big)^{s_0}
		\end{split}
		\end{equation*}
		
		Taking log and multiply by $-1$ on both sides, 
		\begin{equation}\label{step3.3}
		\begin{split}
		&-\log \pi(\Twonorm{(\bB_j - \bBzero_j)\sqrtbSigmaInv} \leq \epsilon_n/s_0  \mbox{ for all } j\in S_0 \vert
		\cap_{i=1}^{q_n}\{\tilde{\bSigma}: 1\leq \lambda_i(\tilde{\bSigma})\leq 1+\epsilon_n^2/q_n\})) \\
		\leq& -\frac{q_n s_0}{2}\log \eta^2\pi + s_0\log\Gamma(\frac{q_n}{2}+1) + q_n s_0\log s_0 + \frac{q_n s_0}{2} \log(1/\epsilon_n^2) - s_0\log\inf_{\Twonorm\bx \leq M_0}
		g_\tau(\bx) \\
		\leq& q_n s_0\log(q_n s_0) + \frac{q_n s_0}{2} \log(1/\epsilon_n^2) + s_0(-\log\inf_{\Twonorm\bx \leq M_0}
		g_\tau(\bx))\\
		\lesssim& q_n s_0\log n + s_0\log p_n \lesssim n\epsilon_n^2
		\end{split}
		\end{equation}
		
		From \eqref{step3.1}, \eqref{step3.2} and \eqref{step3.3}, we get 
		\[
		\pi(\cap_{i=1}^{d}\{\tilde{\bSigma}: 1\leq \lambda_i(\tilde{\bSigma})\leq 1+\epsilon_n^2/q_n\} \mbox{ and }
		\TwoOnenorm{(\bB - \bBzero)\sqrtbSigmaInv} \leq \eta\epsilon_n) 
		\geq e^{-k_4 n \epsilon_n^2/2}
		\]
		for sufficiently large $k_4$.\\
		Also note that by Lemma A.4 in \citet{song2017nearly}, $k_4$ needs to satisfy $k_4 <\min(k_1, k_3)$. Existence of such $k_4$ is ensured by setting $M$ be a sufficiently large constant so that $k_1$ and $k_3$ are large.\\
	\end{proof}
	
	
	\begin{proof} (Theorem \ref{thm:selection_consistency})
	Let $A_n$ be the auxiliary set defined in Theorem 1, then
\begin{align*}
A_n^c = &\{\mbox{at most } \tilde{p} \mbox{ entries } \norm{\bB_j\bSigma^{-1/2}} \mbox{ is larger than } a_n\} \\
\cap  &\{\norm{\bSigma-\bSigma_0} \leq M(\norm\bSigma \lor \norm{\bSigma_0})\epsilon_n \} \\
\cap  &\{\Fnorm{(\bB-\bBzero)\bSigma_0^{-1/2}} \leq M \epsilon_n \}
\end{align*}
From Theorem 1, we have 
$P_0\{\Pi(A_n^c|\bYn)> 1-\exp(-c_1n\epsilon_n^2)\}>1-\exp(-c_2n\epsilon_n^2)$. 
All the following analysis is conditioning on $\Omega_n = \Pi(A_n^c|\bYn)> 1-\exp(-c_1n\epsilon_n^2)$.

The proof has three parts. Part 1 shows that only the case $S_n\supseteq S_0$ needs to be considered.
Part 2 calculates  $\pi(S=S_0|\bYn)$ and $\pi(S=S'|\bYn)$ where $S'\supseteq S_0$. 
Part 3 shows that $\sum_{S'\supseteq S_0, \det{S'\backslash S_0}\geq1}\dfrac{\pi(S=S'|\bYn)}{\pi(S=S_0|\bYn)} \to 0$.\\

\textbf{Part I:} We first show that it suffices to consider $S_n\supseteq S_0$ because $\pi(S_n\nsupseteq S_0 | \bYn) \to 0$.
Note that
\begin{align*}
\pi(S_n\nsupseteq S_0 | \bYn) 
&= \pi(\exists \; j \text{ satisfying } \bBzero_j\neq \bB_j, \norm{\bB_j\sqrtbSigmaInv}\leq a_n | \bYn) \\
& \leq \pi(\norm{(\bB_j-\bBzero_j)\sqrtbSigmaOInv} \geq M\epsilon_n | \bYn) \to 0
\end{align*}
The inequality holds because for $j \text{ satisfying } \bBzero_j\neq \bB_j$ and $\norm{\bB_j\sqrtbSigmaInv}\leq a_n$, 
\begin{align*}
\norm{(\bB_j-\bBzero_j)\sqrtbSigmaOInv} & \geq \norm{\bBzero_j\sqrtbSigmaOInv} - \norm{\bB_j\sqrtbSigmaOInv} \\
& \geq \sqrt{1/b_2}\norm{\bBzero_j} - \sqrt{2b_2/b_1}\norm{\bB_j\sqrtbSigmaInv} \\
& \geq \sqrt{1/b_2}M_1\epsilon_n - \sqrt{2b_2/b_1}a_n \geq M_2 \epsilon_n
\end{align*}
since $\norm{\bSigma_0} \leq b_2$ and $\norm{\bSigma\bSigma_0^{-1}} \leq 2b_2/b_1$ by previous results.\\

\textbf{Part II:} Let $E_1 = \{(\bB, \bSigma): \Fnorm{\bB_1 - \bB_{01}}\leq c_1\epsilon_n, \norm{\bSigma-\bSigma_0} \leq c_2\epsilon_n\}$, and $\underline{\pi(\bB_1 |\bSigma)} = \inf_{(\bB_1, \bSigma)\in E_1} \pi(\bB_1, \bSigma)/\pi(\bSigma)$, 
$\overline{\pi(\bB_1|\bSigma)}  = \sup_{(\bB_1, \bSigma)\in E_1} \pi(\bB_1, \bSigma)/\pi(\bSigma)$.\\

Consider $\pi(S=S_0|\bYn)$, let subscripts ``1" and ``2" denote model $S_0$ and $S_0^c$ respectively.
It is easy to check that $\min_j \norm{\bB_{1j}\sqrtbSigmaInv} > a_n$ for $S_n \supseteq S_0$, so we have
\begin{align*}
\pi(S=S_0 | \bYn)  =& \pi(\TwoInfnorm{\bB_{2}\sqrtbSigmaInv} \leq a_n | \bYn)\\
\propto  & \int \det\bSigma^{-n/2}\exp\{-\dfrac{1}{2}
\sum_{i=1}^{n}(Y_i-X_i\bB)\Sigma^{-1}(Y_i-X_i\bB)^{T}\}
\pi(\bB, \bSigma) \\
& I(\TwoInfnorm{\bB_{2}\sqrtbSigmaInv} \leq a_n) d\bB d\bSigma \\
\geq & \pi(\TwoInfnorm{\bB_{2}\sqrtbSigmaInv} \leq a_n) \int_{E_1} \inf_{\TwoInfnorm{B_{2}\sqrtbSigmaInv} \leq a_n}\det\bSigma^{-n/2} \times \\
&\exp\{-\dfrac{1}{2} \sum_{i=1}^{n}(Y_i-X_i\bB)\Sigma^{-1}(Y_i-X_i\bB)^{T}\}
\pi(\bB_1, \bSigma) d\bB_1 d\bSigma
\end{align*}

\begin{align*}
&\int_{E_1} \inf_{\TwoInfnorm{B_{2}\sqrtbSigmaInv} \leq a_n}\det\bSigma^{-n/2} 
\exp\{-\dfrac{1}{2} \sum_{i=1}^{n}(Y_i-X_i\bB)\Sigma^{-1}(Y_i-X_i\bB)^{T}\}
\pi(\bB_1, \bSigma) d\bB_1 d\bSigma \\
\geq & \underline{\pi(\bB_1 |\bSigma)} \int_{\norm{\bSigma-\bSigma_0} \leq c_2\epsilon_n}  \inf_{\TwoInfnorm{B_{2}\sqrtbSigmaInv} \leq a_n}\det\bSigma^{-n/2} \pi(\bSigma) \times \\ &
\int_{\Fnorm{\bB_1 - \bB_{01}}\leq c_1\epsilon_n}
\exp\{-\dfrac{1}{2}\sum_{i=1}^{n}(Y_i-X_{i1}\bB_1-X_{i2}\bB_2)\Sigma^{-1}(Y_i-X_{i1}\bB_1-X_{i2}\bB_2)^{T}\}
 d\bB_1 d\bSigma \\
\end{align*}
Let
\begin{align*}
SSE(\bB_2,\bSigma) & =  \min_{\bB_1}\FnormSq{(\bY-\bX_1\bB_1-\bX_2\bB_2)\Sigma^{-1/2}} \\
& = \FnormSq{(\bY-\bX_1\hat{\bB_1}-\bX_2\bB_2)\Sigma^{-1/2}} \\
\end{align*}
where $\hat{\bB_1} = (\bX_1^T\bX_1)^{-1}\bX_1^T(\bY-\bX_2\bB_2)$\\

Then we have 
\begin{align*}
& \int_{\Fnorm{\bB_1 - \bB_{01}}\leq c_1\epsilon_n}
\exp\{-\dfrac{1}{2}\sum_{i=1}^{n}(Y_i-X_{i1}\bB_1-X_{i2}\bB_2)\Sigma^{-1}(Y_i-X_{i1}\bB_1-X_{i2}\bB_2)^{T}\}
d\bB_1 \\
= & \exp\{-\dfrac{1}{2} SSE(\bB_2, \bSigma)\} \times \\
& \int_{\Fnorm{\bB_1 - \bB_{01}}\leq c_1\epsilon_n}
\exp\{-\dfrac{1}{2}
tr(\bX_1^T \bX_1(\bB_1-\hat{\bB_1})\bSigma^{-1}(\bB_1-\hat{\bB_1})^T)\}
d\bB_1 \\
\geq & (2\pi)^{-(s_0 q_n/2)} \det{\bSigma}^{s_0/2}\det{\bX_1^T\bX_1}^{-q_n/2}
\exp\{-\dfrac{1}{2} SSE(\bB_2, \bSigma)\} 
\Pr(\Fnorm{\bB_1-\hat{\bB_1}} \leq c\sqrt{s_0 q_n/n})
\end{align*}
where $\bB_1 \sim MatrixNormal(\hat{\bB_1}, \bSigma, (\bX_1^T\bX_1)^{-1})$.\\

Let $T_{\bB} \sim MatrixNormal(\bzero, I_{q_n}, I_{s_0})$, then $\bB_1-\hat{\bB_1} = \Sigma^{1/2}T_{\bB}(\bX_1^T\bX_1)^{-1/2}$. Since $\norm{\Sigma}$ is bounded and $\norm{(\bX_1^T\bX_1)^{-1}} \leq 1/(n\lambda_0)$, we have
\begin{align*}
\Pr(\Fnorm{\bB_1-\hat{\bB_1}} \geq c\sqrt{s_0 q_n/n})
& = \Pr(\Fnorm{\Sigma^{1/2}T_{\bB}(\bX_1^T\bX_1)^{-1/2}} \geq c\sqrt{s_0 q_n/n}) \\
& = \Pr(\FnormSq{T_{\bB}} \geq \tilde{c}s_0 q_n)\\
& = \Pr(\chi_{s_0 q_n}^2 \geq \tilde{c}s_0 q_n) \\
& \leq \exp(-2s_0 q_n)  \text{ if } \tilde{c} \geq 8
\end{align*}
So we have $\Pr(\Fnorm{\bB_1-\hat{\bB_1}} \lesssim \sqrt{s_0 q_n/n}) \to 1$ for large $n$.\\

Let function $K_n(\bB_S):=(\bY-\bX_S\bB_S)^T(I_n-H_{S^c})(\bY-\bX_S\bB_S)$, where $H_{S^c}=\bX_{S^c}(\bX_{S^c}^T\bX_{S^c})^{-1}\bX_{S^c}^T$
\begin{align*}
&\int_{\norm{\bSigma-\bSigma_0} \leq c_2\epsilon_n} 
\inf_{\TwoInfnorm{B_{2}\sqrtbSigmaInv} \leq a_n}
\det{\bSigma}^{-(n-s_0)/2} 
\exp\{-\dfrac{1}{2} SSE(\bB_2, \bSigma)\}
\pi(\bSigma) d\bSigma \\
\geq & \inf_{\TwoInfnorm{B_{2}\sqrtbSigmaOInv} \leq a_n} \int_{\norm{\bSigma-\bSigma_0} \leq c_2\epsilon_n} 
\det{\bSigma}^{-(n-s_0)/2} \exp\{-\dfrac{1}{2}tr((K_n(\bB_2)+\Phi)\bSigma^{-1})\}
\pi(\bSigma) d\bSigma \\
= & \int_{\norm{\bSigma-\bSigma_0} \leq c_2\epsilon_n} 
\dfrac{\det{\Phi}^{\upsilon/2}}{2^{\upsilon q_n /2} \Gamma_{q_n}(\upsilon/2)}
\det{\bSigma}^{-(\upsilon+q_n+1+n-s_0)/2} \exp\{-\dfrac{1}{2}tr((K_n(\bB_2)+\Phi)\bSigma^{-1})\} d\bSigma \\
=& \dfrac{2^{(n-s_0)q_n/2}\det{\Phi}^{\upsilon/2} \Gamma_{q_n}((\upsilon+n-s_0)/2)}
{\det{K_n(\bB_2) + \Phi}^{(\upsilon+n-s_0)/2} {\Gamma_{q_n}(\upsilon/2)}}
\Pr(\norm{T_{\bYn}-\bSigma_0} \leq c_2\epsilon_n)
\end{align*}
where $T_{\bYn} \sim \InvWis(\upsilon+n-s_0, K_n(\bB_2) + \Phi)$ and $\Gamma_{q_n}(\cdot)$ is multivariate gamma function.\\

For simplicity, let $T_{\bYn} \sim \InvWis(\nu_n, \Psi_{Y_n})$, where $\nu_n=\upsilon+n-s_0 \simeq n$,  $\Psi_{\bYn}=K_n(\bB_2) + \Phi$. 
\begin{align*}
\Pr(\norm{T_{\bYn}-\bSigma_0} \geq c_2\epsilon_n |\bYn)
& \leq \E(\norm{T_{\bYn}-\bSigma_0}^2|\bYn)/(c_2^2\epsilon_n^2)\\
& \leq \E(\FnormSq{T_{\bYn}-\bSigma_0}|\bYn)/(c_2^2\epsilon_n^2)\\
& = \FnormSq{\E(T_{\bYn})- \Sigma_0}/(c_2^2\epsilon_n^2) +
    \sum_i\sum_j Var(T_{\bYn,ij})/(c_2^2\epsilon_n^2)
\end{align*}

We first show that if 
$\Pr_{(\bB_0, \Sigma_0)} (\norm{\dfrac{\Psi_{\bYn}}{\nu_n-q_n-1} - \Sigma_0} \geq \dfrac{\epsilon_n}{M\sqrt{q_n}}) \leq \exp(-\tilde{c}n\epsilon_n^2/q_n)$
holds for sufficiently large $M$, then $\FnormSq{\E(T_{\bYn})- \Sigma_0}/(c_2^2\epsilon_n^2) \to 0 \text{ in } P_{(\bB_0,\Sigma_0)}$ and 
$\sum_i\sum_j Var(T_{\bYn,ij})/(c_2^2\epsilon_n^2) \to 0 \text{ in } P_{(\bB_0,\Sigma_0)}$. 

$\FnormSq{\E(T_{\bYn})- \Sigma_0}/\epsilon_n^2 = \FnormSq{\dfrac{\Psi_{\bYn}}{\nu_n-q_n-1} - \Sigma_0}/\epsilon_n^2
\leq \dfrac{q_n}{\epsilon_n^2} \norm{\dfrac{\Psi_{\bYn}}{\nu_n-q_n-1} - \Sigma_0}  \to 0 \text{ in } P_{(\bB_0,\Sigma_0)}$

\begin{equation*}
Var(T_{\bYn}) = 
\begin{cases}
\dfrac{(\nu_n-q_n+1)\Psi_{\bYn,ij}^2 + (\nu_n-q_n-1)\Psi_{\bYn,ii}\Psi_{\bYn,jj}}
{(\nu_n-q_n)(\nu_n-q_n-1)^2(\nu_n-q_n-3)}, & i\neq j \\
\\
\dfrac{2\Psi_{\bYn,ii}^2}{(\nu_n-q_n-1)^2(\nu_n-q_n-3)}, & i = j\\
\end{cases}
\end{equation*}
Because $\Fnorm{\dfrac{\Psi_{\bYn}}{\nu_n-q_n-1} - \Sigma_0} \leq 
	\norm{\dfrac{\Psi_{\bYn}}{\nu_n-q_n-1} - \Sigma_0} \to 0 
	\text{ in } P_{(\bB_0,\Sigma_0)}$, i.e. $\dfrac{\Psi_{\bYn,ij}}{\nu_n-q_n-1} \to \Sigma_{0, ij} \text{ in } P_{(\bB_0,\Sigma_0)}$, $Var(T_{\bYn, ij} | \bYn) = \bigO_{P_{(\bB_0,\Sigma_0)}}(1/n)$. Hence  $\sum_i\sum_j Var(T_{\bYn,ij})/(c_2^2\epsilon_n^2) \leq \bigO_{P_{(\bB_0,\Sigma_0)}}(\frac{q_n^2}{n\epsilon_n^2})\to 0 \text{ in } P_{(\bB_0,\Sigma_0)}$.
	
Now, it suffices to show $Pr_{(\bB_0, \Sigma_0)} (\norm{\dfrac{\Psi_{\bYn}}{\nu_n-q_n-1} - \Sigma_0} \geq \dfrac{\epsilon_n}{M\sqrt{q_n}}) \leq \exp(-\tilde{c}n\epsilon_n^2/q_n)$ for sufficiently large $M$. 

By Lemma 1, let $\tilde{\epsilon_n} = \frac{\epsilon_n}{M\sqrt{q_n}}$, $d=q_n$, $\Sigma = \Sigma_0$, $A =\bepsilonn\sqrtbSigmaO$, $P=I_n-H_1$, $U=-X_2\bB_{2}$,
it is easy to verify that $\tilde{\epsilon_n}\to 0$, $n\tilde{\epsilon_n}^2 \to \infty$, $q_n\prec n\tilde{\epsilon_n}^2$, $\nu_n-q_n-1 \simeq n$ and since $a_n p_n \lesssim (\log n / n)^{1/2}$, we have
\[
\FnormSq{(I_n-H_1)X_2\bB_{2}\sqrtbSigmaOInv} \leq \FnormSq{X_2\bB_2\sqrtbSigmaOInv} \leq np_n^2a_n^2
\lesssim \log n \lesssim n\tilde{\epsilon_n}^2.
\]
Hence,
$\Pr(\norm{\dfrac{S_n}{\nu_n-q_n-1} - \Sigma_0} \geq \tilde{\epsilon_n}) \leq \exp(-\tilde{c}n\tilde{\epsilon_n}^2)$ by Lemma 1. And $\norm{\dfrac{\Phi}{\nu_n-q_n-1}} \leq \Fnorm{\dfrac{\Phi}{\nu_n-q_n-1}} \lesssim \dfrac{q_n}{n} \prec \tilde{\epsilon_n}$. We have $\Pr_{(\bB_0, \Sigma_0)} (\norm{\dfrac{\Psi_{\bYn}}{\nu_n-q_n-1} - \Sigma_0} \geq \dfrac{\epsilon_n}{M\sqrt{q_n}}) \leq \exp(-\tilde{c}n\epsilon_n^2/q_n)$ for sufficiently large $M$. Therefore, $\Pr(\norm{T_{\bYn}-\bSigma_0} \lesssim \epsilon_n) \to 1$ for large $n$.\\

We now calculate $\pi(S=S' | \bYn)$ where $S'\supset S_0$. 
For notation simplicity, we abuse subscripts a little bit.
Now let subscripts ``1", ``2" and ``3" denote model $S_0$, $S'\backslash S_0$ and $(S')^c$ respectively.
\begin{align*}
\pi(S=S' | \bYn) \propto 
& \int \det\bSigma^{-n/2}\exp\{-\dfrac{1}{2}
\sum_{i=1}^{n}(Y_i-X_i\bB)\Sigma^{-1}(Y_i-X_i\bB)^{T}\}
\pi(\bB, \bSigma) \\
& I(\TwoInfnorm{B_{3}\sqrtbSigmaInv} \leq a_n)
I(\min_j \norm{B_{2j}\sqrtbSigmaInv} > a_n) d\bB d\bSigma \\
\lesssim & 
\pi(\TwoInfnorm{B_{3}\sqrtbSigmaInv} \leq a_n, \;
     \min_j \norm{B_{2j}\sqrtbSigmaInv} > a_n) \times \\
 \sup_{\TwoInfnorm{B_{3}\sqrtbSigmaOInv} \lesssim a_n} & \int_{E_1} \det\bSigma^{-n/2} \exp\{-\dfrac{1}{2} \sum_{i=1}^{n}(Y_i-X_i\bB)\Sigma^{-1}(Y_i-X_i\bB)^{T}\}
\pi(\bB_1, \bSigma) d\bB_1 d\bSigma
\end{align*}

\begin{align*}
& \int_{E_1} \det\bSigma^{-n/2} \exp\{-\dfrac{1}{2} \sum_{i=1}^{n}(Y_i-X_i\bB)\Sigma^{-1}(Y_i-X_i\bB)^{T}\}
\pi(\bB_1, \bSigma) d\bB_1 d\bSigma \\
\leq & \overline{\pi(\bB_1|\bSigma)} 
\exp\{-\dfrac{1}{2} SSE(\bB_2, \bB_3, \bSigma)\} \times \\
& \int \det{\bSigma}^{-n/2} 
\exp\{-\dfrac{1}{2}
tr(\bX_1^T \bX_1(\bB_1-\tilde{\bB_1})\bSigma^{-1}(\bB_1-\tilde{\bB_1})^T)\}
d\bB_1 \\
= & \overline{\pi(\bB_1|\bSigma)}
(2\pi)^{-(s_0 q_n/2)}\det{\bX_1^T\bX_1}^{-q_n/2} 
 \int \det{\bSigma}^{-(n-s_0)/2} \exp\{-\dfrac{1}{2}SSE(\bB_2,\bB_3,\bSigma)\} \pi(\bSigma) d\bSigma \\
= & 
\overline{\pi(\bB_1|\bSigma)}
(2\pi)^{-(s_0 q_n/2)}\det{\bX_1^T\bX_1}^{-q_n/2} 
\dfrac{2^{(n-s_0)q_n/2}\det{\Phi}^{\upsilon/2} \Gamma_{q_n}((\upsilon+n-s_0)/2)}
{\det{K_n(\bB_2, \bB_3) + \Phi}^{(\upsilon+n-s_0)/2} {\Gamma_{q_n}(\upsilon/2)}} 
\end{align*}

\textbf{Part III:} Now we want to show $\sum_{S'\supseteq S_0, \det{S'\backslash S_0}\geq1}\dfrac{\pi(S=S'|\bYn)}{\pi(S=S_0|\bYn)} \to 0$.
\begin{align*}
& \dfrac{\pi(S=S' | \bYn)}{\pi(S=S_0 | \bYn)} \lesssim 
\dfrac{\overline{\pi(\bB_1|\bSigma)} }{\underline{\pi(\bB_1 |\bSigma)}}
\dfrac{\pi(\TwoInfnorm{B_{(S')^c}\sqrtbSigmaInv} \leq a_n, \; \min_j \norm{({B_{S'\backslash S_0}\sqrtbSigmaInv})_j} > a_n)}
{\pi(\TwoInfnorm{B_{S_0^c}\sqrtbSigmaInv} \leq a_n)} \times \\
& \dfrac{\sup_{\TwoInfnorm{B_{S_0^c}\sqrtbSigmaOInv} \leq a_n}
{\det{K_n(\bB_{S_0^c}) + \Phi}^{(\upsilon+n-s_0)/2}}}
{\inf_{\TwoInfnorm{B_{S'^c}\sqrtbSigmaOInv} \leq a_n}
{\det{K_n(\bB_{S'\backslash S_0}, \bB_{S'^c}) + \Phi}^{(\upsilon+n-s_0)/2}}}
\end{align*}
Firstly, \[
\dfrac{\overline{\pi(\bB_1|\bSigma)} }{\underline{\pi(\bB_1 |\bSigma)}} \leq l_n^{s_0} \text{by \eqref{cond:flatness_g}}\] 
And we also have
\begin{align*}
& \dfrac{\pi(\TwoInfnorm{B_{(S')^c}\sqrtbSigmaInv} \leq a_n, \; \min_j \norm{({B_{S'\backslash S_0}\sqrtbSigmaInv})_j} > a_n)}
{\pi(\TwoInfnorm{B_{S_0^c}\sqrtbSigmaInv} \leq a_n)} \\
= & \dfrac{\pi(\min_j \norm{({B_{S'\backslash S_0}\sqrtbSigmaInv})_j} > a_n)}
{\pi(\TwoInfnorm{B_{(S'\backslash S_0)^c}\sqrtbSigmaInv} \leq a_n)} \\
\leq & \big(\dfrac{p_n^{-(1+u)}}{1-p_n^{-(1+u)}} \big)^{\det{S'\backslash S_0}}
\end{align*}

Next, we examine the behavior of $\dfrac{\det{K_n(\bB_{S_0^c}) + \Phi}}{\det{K_n(\bB_{S'^c}) + \Phi}}$ under constraints $\TwoInfnorm{B_{S_0^c}\sqrtbSigmaOInv} \leq a_n$ and $\TwoInfnorm{B_{S'^c}\sqrtbSigmaOInv} \leq a_n$.
It is easy to verify that 
\[
\dfrac{\det{K_n(\bB_{S_0^c}) + \Phi}}{\det{K_n(\bB_{S'^c}) + \Phi}} \leq 
\dfrac{\det{K_n(\bB_{S_0^c}) + \Phi}}{\det{K_n(\bB_{S'^c})}}
\]
where 
\begin{align*}
K_n(\bB_{S_0^c}) = & (\bYn-\bX_{S_0^c}\bB_{S_0^c})^T(I_n-H_{S_0})(\bYn-\bX_{S_0^c}\bB_{S_0^c})\\
= & \Sigma_0^{1/2}\bepsilonn^T(I_n-H_{S_0})\bepsilonn \Sigma_0^{1/2} +
(\bX_{S_0^c}\bB_{S_0^c})^T(I_n-H_{S_0})(\bX_{S_0^c}\bB_{S_0^c}) \\
& - 2 \Sigma_0^{1/2}\bepsilonn^T(I_n-H_{S_0})(\bX_{S_0^c}\bB_{S_0^c})\\
K_n(\bB_{S'^c}) = & (\bYn-\bX_{S'^c}\bB_{S'^c})^T(I_n-H_{S'})(\bYn-\bX_{S'^c}\bB_{S'^c})\\
= & \Sigma_0^{1/2}\bepsilonn^T(I_n-H_{S'})\bepsilonn \Sigma_0^{1/2} \\
& +
(\bX_{S_0}\bB_{S_0}-\bX_{S'^c}\bB_{S'^c})^T(I_n-H_{S'})(\bX_{S_0}\bB_{S_0}-\bX_{S'^c}\bB_{S'^c}) \\
& - 2 \Sigma_0^{1/2}\bepsilonn^T(I_n-H_{S'})(\bX_{S_0}\bB_{S_0}-\bX_{S'^c}\bB_{S'^c})
\end{align*}

Because $
\FnormSq{\bX_{S_0^c} \bB_{S_0^c}} \leq \FnormSq{\bX_{S_0^c}} \FnormSq{\bB_{S_0^c}} \lesssim n p_n^2 a_n^2 \lesssim n\epsilon_n^2  
\prec n $ and 
\begin{align*}
\Fnorm{\Sigma_0^{1/2}\bepsilonn^T(I_n-H_{S_0})(\bX_{S_0^c}\bB_{S_0^c})} &
\lesssim \norm{\bepsilonn}\Fnorm{\bX_{S_0^c}\bB_{S_0^c}} \\
& \lesssim \sqrt{n}\sqrt{n\epsilon_n^2} \text{ with probability } \to 1 \\
& \prec n
\end{align*}
we have $\det{K_n(\bB_{S_0^c})} \simeq \det{\bepsilonn^T(I_n-H_{S_0})\bepsilonn}\det{\Sigma_0}$.

Similarly, $ \FnormSq{(\bX_{S_0}\bB_{S_0}-\bX_{S'^c}\bB_{S'^c})} \prec n $
and 
$
\Fnorm{\Sigma_0^{1/2}\bepsilonn^T(I_n-H_{S'})(\bX_{S_0}\bB_{S_0}-\bX_{S'^c}\bB_{S'^c})} \prec n 
$, so
 we have $\det{K_n(\bB_{S'^c})} \simeq \det{\bepsilonn^T(I_n-H_{S'})\bepsilonn}\det{\Sigma_0}$. \\
 
By the bounds for eigenvalues in \citet{vershynin2012introduction}, with probability going to 1,
\begin{align*}
& n-s_0 \lesssim \lambda_{min}(\bepsilonn^T(I_n-H_{S_0})\bepsilonn) \leq \lambda_{max}(\bepsilonn^T(I_n-H_{S_0})\bepsilonn) \lesssim n-s_0\\
& n-s_0 \lesssim \lambda_{min}(\bepsilonn^T(I_n-H_{S'})\bepsilonn) \leq \lambda_{max}(\bepsilonn^T(I_n-H_{S'})\bepsilonn) \lesssim n-s_0.
\end{align*}
Therefore,
\[
\dfrac{\det{K_n(\bB_{S_0^c}) + \Phi}}{\det{K_n(\bB_{S'^c})}} \simeq \dfrac{\det{\bepsilonn^T(I_n-H_{S_0})\bepsilonn}}{\det{\bepsilonn^T(I_n-H_{S'})\bepsilonn}} \lesssim (\dfrac{n-s_0}{n-s'})^{q_n}.
\]

Note that $(\dfrac{n-s'}{n-s_0})^{q_n(\upsilon + n-s_0)} \simeq \big((1-\dfrac{s'-s_0}{n-s_0})^{n-s_0}\big)^{q_n} \geq \exp(-c(s'-s_0)q_n) \geq \exp(-c(s'-s_0) \log p_n)$ because $(1-\dfrac{s'-s_0}{n-s_0})^{n-s_0} \geq \exp(-c(s'-s_0))$ for large $n$, $c>1$ and $q_n \lesssim \log p_n$.

Therefore, 
\[
\dfrac{\sup_{\TwoInfnorm{\bB_{S_0^c}\sqrtbSigmaOInv} \leq a_n}
	{\det{K_n(\bB_{S_0^c}) + \Phi}^{(\upsilon+n-s_0)/2}}}
{\inf_{\TwoInfnorm{\bB_{S'^c}\sqrtbSigmaOInv} \leq a_n}
	{\det{K_n(\bB_{S'^c}) + \Phi}^{(\upsilon+n-s_0)/2}}} \leq \exp(c\det{S'\backslash S_0} \log p_n).
\]

Combining the above parts, we have 
$ \dfrac{\pi(S=S' | \bYn)}{\pi(S=S_0 | \bYn)} \leq l_n^{s_0}(1+p_n^{c-(1+u)})^{\det{S'\backslash S_0}}$ and

$\sum_{S'\supseteq S_0, \det{S'\backslash S_0}\geq1}\dfrac{\pi(S=S'|\bYn)}{\pi(S=S_0|\bYn)} \leq l_n^{s_0}((1+p_n^{-(1+u-c)})^{p_n}-1) \simeq l_n^{s_0}p_n^{-(u-c)}$.

By \eqref{cond:flatness_g}, we can get $\sum_{S'\supseteq S_0, \det{S'\backslash S_0}\geq1}\dfrac{\pi(S=S'|\bYn)}{\pi(S=S_0|\bYn)} \to 1$ for $1\leq c < u$.

	\end{proof}
	
	\begin{proof} (Theorem \ref{thm: scalemixureprior})
		It is easy to verify that when $\bB$ follows the given distribution, $g_\tau(\cdot)$ takes the form $g_\tau(\cdot) = g(\cdot/\sqrt\tau)/\tau^{q_n/2}$, where $\tau = \tau_n$ and
		\[
		g(\bx) = \int_{0}^{\infty}(2\pi)^{-q_n/2}\xi^{-q_n/2} \exp(-\frac{\TwonormSq\bx}{2\xi})
		\cdot K\xi^{-r}L(\xi)d\xi .
		\]
		
		We first show that \eqref{cond:highlynear0} holds when $L(\xi)$ has upper bound.
		When $L(\xi)\leq C_{12}$, we have 
		\[
		g(\bx) \leq C_{12}K2^{r-1}\pi^{-q_n/2}\Gamma(\frac{d}{2}+r-1)\Twonorm\bx^{-(d+2r-2)}.
		\]
		And by Lemma \eqref{volumn},

		\begin{equation*}
		\begin{split}
		\int_{\Twonorm{\bx} \geq a_n} g_\tau(\bx)d\bx 
		&= \int_{\Twonorm{\bx} \geq a_n} g(\bx/\sqrt\tau_n)/\tau_n^{q_n/2}d\bx \\
		&= \int_{\Twonorm{\bz} \geq a_n/\sqrt{\tau_n}} g(\bz)d\bz \quad 
		(\bz = \bx/\sqrt\tau_n)\\
		&\leq C_{12}K2^{r-1}\pi^{-q_n/2}\Gamma(\frac{q_n}{2}+r-1)\int_{\Twonorm{\bz} \geq a_n/\sqrt{\tau_n}} \Twonorm\bz^{-(d+2r-2)} d\bz\\
		&\leq \frac{C_{12}K2^{r-1}}{2(r-1)}(\frac{d}{2}+r-1)^r\big(\frac{\tau_n}{a_n^2}\big)^{r-1}\\
		&\leq p_n^{-(1+u)} \quad (0<u<u')
		\end{split}
		\end{equation*}
		since $q_n/2+r-1\lesssim \log p_n \prec p_n^{u'-u}$ for any $0<u<u'$ and $(\tau_n/a_n^2)^{r-1} \lesssim p_n^{-(1+u')}$. $L(\xi)\leq 1$ is a special case of $L(\xi)\leq C_{12}$.\\
		
		Now, we want to show that \eqref{cond:fattail} holds with $\log \tau_n \gtrsim -\log p_n$. It is obvious that $g(\bx)$ is a decreasing function of $\Twonorm\bx$, so
		$\inf_{\Twonorm\bx \leq M_0} g_{\tau_n}(\bx) \geq g_{\tau_n}(\bx)\lvert_{\Twonorm\bx=M_0}$. 
		When $L(\xi) \geq 1-C_{11}\xi^{-t}$, 
		\[
		g(\bx) \geq K2^{r-1}\pi^{-q_n/2}\Gamma(\frac{q_n}{2}+r-1)\Twonorm\bx^{-(d+2r-2)}
		\big(1-C_{11}2^t \frac{\Gamma(\frac{d}{2}+r+t-1)}{\Gamma(\frac{d}{2}+r-1)} \Twonorm\bx^{-2t}\big).
		\]
		Since $\tau_n \lesssim a_n^2p_n^{-(1+u')/(r-1)}$, 
		\[
		C_{11} 2^t \frac{\Gamma(\frac{q_n}{2}+r+t-1)}{\Gamma(\frac{q_n}{2}+r-1)}\Twonorm{M_0/\sqrt{\tau_n}}^{-2t} \leq C_{11}M_0^{-2t} [(d+2(r+t-1))\tau_n]^t \to 0.
		\]
		So we have 
		\begin{equation*}
		\begin{split}
		\log g_{\tau_n}(\bx)\lvert_{\Twonorm\bx=M_0} &\geq \log \big(
		\tau_n^{-q_n/2} \tilde{K}2^{r-1}\pi^{-q_n/2}\Gamma(\frac{q_n}{2}+r-1)\Twonorm{M_0/\sqrt{\tau_n}}^{-(q_n+2r-2)}
		\big) \\
		&\geq constant - d\log(\sqrt{\pi}M_0) + (r-1)\log \tau_n \\
		& \gtrsim -\log p_n
		\end{split}
		\end{equation*}
		because $d\lesssim \log p_n$ and $\log\tau_n \gtrsim -\log p_n$.\\
		When $L(\xi) \geq C_{21}\xi^{-t_1}$, $g(\bx) \geq KC_{21}2^{t_1+r-1}\pi^{-d/2}\Gamma(\frac{d}{2}+t_1+r-1)\Twonorm\bx^{-(d+2(t_1+r-1))}$, the rest follows the above inequalities.\\
	\end{proof}

	\begin{corollary}
		Polynomial-tailed distribution: Student's t-distribution, TPBN, HIB, GDP and Horseshoe+ satisfy either condition (1) or (2) in Theorem \eqref{thm: scalemixureprior}.
	\end{corollary}
	
	\begin{proof}
		For t-distribution, let $L(\xi) = \exp(-a/\xi)$, $a>0$, and $1-a\xi^{-1}\leq \exp(-a/\xi) \leq 1$ for $\xi>0$.\\
		For TPBN distribution, let $L(\xi) = (\xi/(1+\xi))^{a+u} \leq1$, $a,u>0$, and by Bernoulli inequality, $(\xi/(1+\xi))^{a+u} = (1+1/\xi)^{-(a+u)} \geq 1-(a+u)\xi^{-1}$ for $\xi>0$. Note that Horseshoe and NEG are special cases of TPBN with $a=u=1/2$ and $ u=1$, respectively.\\
		For HIB distribution, let \\
		$
		L(\xi) = (1\lor\phi^2)e^s(\xi/(1+\xi))^{a+u}\exp(-\frac{s}{1+\xi})(\phi^2+\frac{1-\phi^2}{1+\xi})^{-1}, \quad a,u,\phi^2>0, s\in\mathbb{R}.
		$
		Because $1\land\frac{1}{\phi^2}\leq (\frac{1-\phi^2}{1+\xi})^{-1} \leq 1\lor\frac{1}{\phi^2}$, we have $1-(a+u)\xi^{-1}\leq L(\xi)\leq (\phi^2\lor\frac{1}{\phi^2}) e^s$.\\
		For GDP distribution, $L(\xi) = \int_{0}^{\infty} t^a\exp(-t-\eta\sqrt{2t/\xi})dt$. Because $1-\sqrt{2}\eta t^{1/2} \xi^{-1/2} \leq \exp(-\eta\sqrt{2t/\xi}) \leq 1$, we have $1-\sqrt{2}\eta\frac{\Gamma(a+3/2)}{\Gamma(a+1)}\xi^{-1/2} \leq L(\xi)/\Gamma(a+1) \leq 1 $.\\
		For Horseshoe+ distribution, $L(\xi) = \xi^{3/4}(\xi-1)^{-1}\log\xi/4$, $\xi>1$. Note that $\log\xi /4 = \log\xi^{1/4}\leq \xi^{1/4}-1$, so $L(\xi)\leq 1$. And $\log\xi\geq 1-\xi^{-1}>0$, so $L(\xi) \geq \xi^{-1/4}/4$.\\
	\end{proof}

	\begin{lemma} \label{Lemma1}
		Let $A$ be a $n\times d$ random matrix with independent rows $A_i \sim \Nor_d(0,\bSigma)$. If $\epsilon_n \to 0$ and $n\epsilon_n^2 \to \infty$ as $n \to \infty$ and $d \prec n\epsilon_n^2$, then for any $n\times n$ projection matrix $P$ with rank $r\simeq n$ and any $n\times d$ fixed matrix $U$ with $\FnormSq{PU\sqrtbSigmaInv} \lesssim n\epsilon_n^2$, we have the following inequalities for some constants $K,c>0$ and sufficiently large $n$,
		\begin{align*}
		&Pr(\norm{\frac{1}{r}A^T P A - \bSigma} \geq K\norm\bSigma\epsilon_n) \leq e^{-c n\epsilon_n^2} \\
		&Pr(\norm{\frac{1}{r}(A+U)^T P (A+U) - \bSigma} \geq 2K\norm\bSigma\epsilon_n) \leq e^{-c n\epsilon_n^2/2}
		\end{align*}
	\end{lemma}
	
	\begin{proof}
		Let $P=Q^T \Lambda Q$ be a spectral decomposition of P, where the first $r$ diagonal elements of $\Lambda$ are 1 and the rest $n-r$ elements are 0. Because $Q$ is orthogonal, $Z=QA$ is a $n\times d$ matrix with independent rows $Z_i\sim \Nor_{d}(0,\bSigma)$ and $A^T P A = Z^T \Lambda Z = \sum_{i=1}^{r}Z_i Z_i^T$, where $Z_i^T$ is the $i$th row of $Z$. By Theorem 5.39 and Remark 5.40 in \citet{vershynin2012introduction},
		\[
		Pr(\norm{\frac{1}{r}\sum_{i=1}^{r}Z_i Z_i^T - \bSigma} \geq \max(\delta,\delta^2)\norm\bSigma)  \leq 2e^{-ct^2} \quad \text{where } \delta=C\sqrt{d/r} + t/\sqrt{r}
		\]
		Let $t=\sqrt{n}\epsilon_n$, the first part is proved because $\max(\delta,\delta^2) \leq K\epsilon_n$ for large $K$.
		
		Now we prove the second inequality. Let $E = A\sqrtbSigmaInv$, then $E$ has iid standard normal entries. It suffices to show	
		\[
		Pr(\norm{\frac{1}{r}(E+U\sqrtbSigmaInv)^T P (E+U\sqrtbSigmaInv) - I_d} \geq 2K\epsilon_n) \leq e^{-c n\epsilon_n^2/2},
		\]
		since $\norm{\frac{1}{r}(A+U)^T P (A+U) - \bSigma} \leq \norm{\frac{1}{r}(E+U\sqrtbSigmaInv)^T P (E+U\sqrtbSigmaInv) - I_{q_n}} \norm\bSigma$
		
		Because $\FnormSq{PU\sqrtbSigmaInv} \lesssim n\epsilon_n^2$, we get $\frac{1}{r}\Fnorm{PU\sqrtbSigmaInv} \prec \epsilon_n$. By first part of Lemma \eqref{Lemma1}, $Pr(\norm{\frac{1}{n}E^T E-I_d} \geq K\epsilon_n) \leq e^{-cn\epsilon_n^2}$ and by the triangle inequality $\norm{\frac{1}{n}E^T E-I_d} \geq \frac{1}{n}\norm{E^T E}-1$, we have $Pr(\norm{E^T E}\geq 2n) \leq e^{-cn\epsilon_n^2}$. Conditioned on $\norm{E^T E}\leq 2n$, $\frac{2}{r}\norm{E^T PU\sqrtbSigmaInv} \leq \frac{2}{r}\norm{E^T E}^{1/2}\Fnorm{PU\sqrtbSigmaInv} \leq K\epsilon_n/2$ for sufficiently large $K$. So we have
		\begin{equation*}
		\begin{split}
		&Pr(\norm{\frac{1}{r}(E+U\sqrtbSigmaInv)^T P (E+U\sqrtbSigmaInv) - I_d}  \geq 2K\epsilon_n)\\
		\leq& 
		Pr(\frac{1}{r}\Fnorm{PU\sqrtbSigmaInv} + \frac{2}{r}\norm{E^T PU\sqrtbSigmaInv}  +
		\norm{\frac{1}{r}E^T P E - I_d} \geq 2K\epsilon_n | \norm{E^T E}\leq 2n
		) \\
		& + Pr(\norm{E^T E}\geq 2n)	\\
		\leq&  Pr(	\norm{\frac{1}{r}E^T P E - I_d} \geq K\epsilon_n) + e^{-cn\epsilon_n^2} \\
		\leq& e^{-cn\epsilon_n^2/2} \mbox{ for large } n
		\end{split}
		\end{equation*}
	\end{proof}

	\begin{lemma}\label{Lemma2}
		For compatible matrices $A$ and $B$, we have $\det{tr(AB)} \leq \TwoInfnorm{A^T} \TwoOnenorm{B}\label{Lemma2.1}$. Further, if $B$ is a square matrix, then 
		$\TwoOnenorm{AB} \leq \sqrt{\lambda_{\max}(B B^T)} \TwoOnenorm{A}
		\label{Lemma2.2}$
	\end{lemma}
	
	\begin{proof}
		\[
		tr(AB) \leq \sum_{j=1}^{m} \sum_{i=1}^{n} \det{a_{ij} b_{ji} }
		\leq \sum_{j=1}^{m}  \sqrt{\sum_{i=1}^{n} a_{ij}^2\sum_{i=1}^{n} b_{ji}^2}
		\leq \TwoInfnorm{A^T} \sum_{i=1}^{n} \TwoOnenorm{B}
		\]
		\[
		\TwoOnenorm{AB} = \sum_{i=1}^{n} \sqrt{A_i B B^T A_i^T} 
		\leq \sum_{i=1}^{n} \sqrt{\lambda_{\max}(B B^T)} \Twonorm{A_i}
		= \sqrt{\lambda_{\max}(B B^T)} \TwoOnenorm{A}
		\]
	\end{proof}

	\begin{lemma} \label{volumn}
		For a $d$-dimentional vector $\bx$,
		\[
		\int_{\Twonorm{\bx} \geq a}\Twonorm\bx^{-(d+k)}d\bx =  \frac{2\pi^{d/2}a^{-k}}{k\Gamma(d/2)}
		\quad k,a>0
		\]
	\end{lemma}
	
	\begin{proof}
		The result is immediate by polar coordinate transformation in \citet{scott2015multivariate}, 
		\[
		\int_{\Twonorm{\bx} \geq a}\Twonorm\bx^{-(d+k)}d\bx
		= \big(\int_{a}^{\infty}r^{-k-1}dr \big)
		\big(\prod_{i=1}^{d-2} \int_{-\pi/2}^{\pi/2} \cos^{d-i-1}\theta_i d\theta_i \big)
		\big(\int_{0}^{2\pi}1 d\theta_{d-1}\big)
		\]
		and $\int_{-\pi/2}^{\pi/2} \cos^{d-i-1}\theta d\theta
		=\pi^{1/2}\frac{\Gamma((d-i)/2)}{\Gamma((d-i+1)/2)}$.
		
	\end{proof}

\end{document}